\documentclass[11pt]{amsart}

\usepackage{a4wide, amsmath, amsfonts, amssymb, mathrsfs, amsthm, bbm}
\usepackage[hyperfootnotes=false]{hyperref}
\usepackage{xcolor}

\numberwithin{equation}{section}
\usepackage{enumitem}

\newtheorem{theorem}{Theorem}[section]
\newtheorem{lemma}[theorem]{Lemma}
\newtheorem{corollary}[theorem]{Corollary}
\newtheorem{remark}[theorem]{Remark}
\newtheorem{proposition}[theorem]{Proposition}
\newtheorem{definition}[theorem]{Definition}

\newtheorem{assumption}[theorem]{Assumption}

\allowdisplaybreaks[2]

\newcommand{\dd}{\,\mathrm{d}}
\renewcommand{\d}{\mathrm{d}}
\renewcommand{\epsilon}{\varepsilon}
\newcommand{\R}{\mathbb{R}}
\newcommand{\N}{\mathbb{N}}
\renewcommand{\P}{\mathbb{P}}
\newcommand{\E}{\mathbb{E}}
\newcommand{\indicator}[1]{\mathbbm{1}_{#1}}
\newcommand{\Id}{\,\mathrm{d}}
\newcommand{\norm}[1]{\left\lVert#1\right\rVert}
\newcommand{\qv}[1]{\langle #1 \rangle}
\newcommand{\testfunctions}[1]{C_c^{\infty}(#1)}
\newcommand{\cF}{\mathcal{F}}

\title{Hegselmann--Krause model with environmental noise}

\author[Chen]{Li Chen}
\address{Li Chen, University of Mannheim, Germany}
\email{chen@uni-mannheim.de}

\author[Nikolaev]{Paul Nikolaev}
\address{Paul Nikolaev, University of Mannheim, Germany}
\email{pnikolae@mail.uni-mannheim.de}

\author[Pr{\"o}mel]{David J. Pr{\"o}mel}
\address{David J. Pr{\"o}mel, University of Mannheim, Germany}
\email{proemel@uni-mannheim.de}

\date{\today}

\begin{document}

\begin{abstract}
  We study a continuous-time version of the Hegselmann--Krause model describing the opinion dynamics of interacting agents subject to random perturbations. Mathematically speaking, the opinion of agents is modelled by an interacting particle system with a non-Lipschitz continuous interaction force, perturbed by idiosyncratic and environmental noises. Sending the number of agents to infinity, we derive a McKean--Vlasov stochastic differential equation as the limiting dynamic, by establishing propagation of chaos for regularized versions of the noisy opinion dynamics. To that end, we prove the existence of a unique strong solution to the McKean--Vlasov stochastic differential equation as well as well-posedness of the associated non-local, non-linear stochastic Fokker--Planck equation.
\end{abstract}

\maketitle




\section{Introduction}

The theory of opinion dynamics has been well-studied since the 1950s, but over the last few decades, the modeling of opinion dynamics has become a rapidly growing area of research. With the rapid development of the internet and social networks, we have observed significant changes in how opinion dynamics evolve and by what sources they are effected. For example, previous generations were heavily influenced by their geographically nearest social group, but nowadays social networks play a dominant role for expressing and sharing opinions, enabling more people than ever to do so from anywhere in the world.  Consequently,  the significance of the geographical distance as a factor in shaping public opinion diminishes. Instead, each citizen has a personal filter bubble \cite{fake_news_ideology}, which affects and in the same time shifts with the opinion. This phenomena is described by so called bounded confidence opinion dynamics. For an overview of opinion dynamics we refer to the surveys \cite{LORENZ_2007_bounded_confidence_survey, noorazar2020recent}.

In the present paper we study the Hegselmann--Krause model (HK model) \cite{hegselmann2002opinion}, which belongs to the class of bounded confidence opinion dynamics. More precisely, we focus on a version of the HK model where the opinions \(\mathbf{X}^N := (X^{i}, i = 1,\dots, N)\) of $N$ agents are subject to idiosyncratic noises as well as an environmental noise, i.e., we consider for \(i=1,\ldots,N \) the particle system
\begin{equation}\label{eq:HR model intro}
  \d X_t^{i} = -\frac{1}{N-1} \sum\limits_{\substack{j=1 \\j \neq i }}^N k_{\scriptscriptstyle{HK}}(X_t^{i} -X_t^{j}) \dd t + \sigma(t,X_t^{i}) \dd B_t^{i} + \nu \dd W_t,
  \quad   X_0^{i} = \zeta_i,
\end{equation}
for $t\geq 0$, where \(X_t^{i}\) is the \(i\)-th agent's opinion at time \(t\), \(k_{\scriptscriptstyle{HK}}(x):=\indicator{[0,R]}(|x|)x \) is the (non-Lipschitz) interaction force between the agents, \(\sigma\colon  [0,T] \times \R \mapsto \R \) and $\nu>0$ some smooth diffusion coefficients, \(B=((B_t^{i}, t \ge 0 ) , i \in \N )\) is a sequence of one-dimensional independent Brownian motions, \(W=(W_t,t\ge 0)\) is a Brownian motion independent of \(B\), and \((\zeta_i, i \in \N)\) is the i.i.d. sequence of initial values independent of all Brownian motions. The local interaction kernel~$k_{\scriptscriptstyle{HK}}$ represents the insight of bounded confidence opinion dynamics that opinions are only influenced in a bounded domain. In the HK model~\eqref{eq:HR model intro}, the idiosyncratic noises \(B=((B_t^{i}, t \ge 0 ) , i \in \N )\) describe the individual random effects on each agent's opinion and the environmental noise  \((W_t, t \ge 0)\) captures external effects on the agents' opinions. For a more detailed discussion on different types of noises in HK models we refer to \cite{chen2019heterogeneous} and the references therein. Additionally, we will recall some background information on the HK model at the beginning of Section~\ref{sec:HK model}.

Our goal is to establish propagation of chaos of the Hegselmann--Krause model with environmental noise. More precisely, we show that regularized versions of the particle systems~\eqref{eq:HR model intro} converge (in a suitable sense) to the McKean--Vlasov stochastic differential equation (SDE)
\begin{align}\label{eq: introduction_mean_field_trajectories_with_common_noise}
  \begin{cases}
  \Id Y_t = -  (k_{\scriptscriptstyle{HK}} * \rho_t)(Y_t) \Id t + \sigma(t,Y_t) \Id B_t^{1} + \nu \Id W_t, \quad Y_0=X_0^{1} , \\
  \rho_t \; \mathrm{is \, the \, conditional \, density \, of} \, Y_t \, \mathrm{given} \, \cF_t^W,
  \end{cases}
\end{align}
which comes with the associated stochastic non-linear, non-local Fokker--Planck equation
\begin{equation}\label{eq: introduction_hk_spde}
  \Id \rho_t = \frac{\dd^2}{\dd x^2}  \Big( \frac{\sigma_t^2 + \nu_t^2 }{2} \rho_t \Big) \Id t + \frac{\dd}{\dd x} ((k_{\scriptscriptstyle{HK}} *\rho_t)\rho_t)  \Id t - \nu \frac{\dd}{\dd x} \rho_t \Id W_t,\quad t\geq 0.
\end{equation}
Let us remark that equation~\eqref{eq: introduction_hk_spde} is a non-local, non-linear stochastic partial differential equation (SPDE), where the stochastic term is a consequence of the environmental noise \(W=(W_t, t \ge 0)\). Indeed, as we shall see, if the number of agents tends to infinity the effect of the idiosyncratic noises averages out, but the environmental noise does not. Moreover, in contrast to many recent works like \cite{CoghiFlandoli2016,Coghi2019,Hammersley2020,Coghi2020,Bailleul2021} on interacting particle systems with environmental noise, we avoid the measure-valued setting and deal with the conditional McKean--Vlasov SDE and the associated Fokker--Planck equations by directly analysing the densities.

Our first contribution is to prove the well-posedness of the non-local, non-linear stochastic Fokker--Planck equation~\eqref{eq:  introduction_hk_spde}. The main challenge in proving existence and uniqueness results for \eqref{eq: introduction_hk_spde} is the non-linear term \((k_{\scriptscriptstyle{HK}} *\rho_t)\rho_t \) since this prevents us from applying known results  in the existing literature on SPDEs, such as those found in textbooks~\cite{Krylov1999AnAA,LiuWei2015SPDE}, which consider the well-studied case of linear SPDEs. In the case of non-linear SPDEs, one needs to take advantage of the specific structure of the considered SPDE to employ a fixed point argument, cf. e.g. the recent work \cite{HuangQui2021}. In this line of research, we establish local existence and uniqueness of a weak solution to the non-local, non-linear Fokker--Planck equation~\eqref{eq: introduction_hk_spde}. Additionally, we show global well-posedness of the Fokker--Planck equation~\eqref{eq: introduction_hk_spde} assuming a sufficiently large diffusion coefficient or a sufficiently small \(L^2\)-norm of the initial value.

Our second contribution is to prove the existence of a unique strong solution to the McKean--Vlasov stochastic differential equation~\eqref{eq: introduction_mean_field_trajectories_with_common_noise}, which is essential for showing propagation of chaos towards to limiting Fokker--Planck equation~\eqref{eq: introduction_hk_spde}. To obtain the well-posedness of the McKean--Vlasov SDE~\eqref{eq: introduction_mean_field_trajectories_with_common_noise}, a central insight is to introduce suitable stopping times to ensure sufficient temporary regularity such that a backward stochastic partial differential equation (BSPDE) associated to \eqref{eq: introduction_mean_field_trajectories_with_common_noise} possesses a classical solution, cf. e.g. \cite{DuTang2013}. As a result, a duality argument in combination with the It{\^o}--Wentzell formula allows us to deduce the existence of a unique strong solution to the McKean--Vlasov equation~\eqref{eq: introduction_hk_spde}. Let us point out that the aforementioned existence and uniqueness results for the stochastic Fokker--Planck equation~\eqref{eq: introduction_hk_spde} and the McKean--Vlasov equation~\eqref{eq: introduction_mean_field_trajectories_with_common_noise} also hold for general interaction forces \(k \in L^1(\R) \cap L^2(\R)\) and non-degenerate diffusion coefficient \(\sigma \in C^3(\R)\).

Our third contribution is to establish propagation of chaos for regularized versions of the particle systems~\eqref{eq:HR model intro} with environmental noise, which verifies that \eqref{eq: introduction_hk_spde} is, indeed, the macroscopic (density based) model corresponding to the microscopic (agent based) opinion dynamics described by the Hegselmann--Krause model with environmental noise. For uniformly Lipschitz interaction forces, propagation of chaos with environmental noise has been showed by Coghi and Flandoli \cite{CoghiFlandoli2016} by utilizing sharp estimates in Kolmogorov's continuity theorem and properties of measure-valued solutions of the associated stochastic Fokker--Planck equation. In particular, the approach is from an SDE level to an SPDE level~\cite{CoghiFlandoli2016,Coghi2019,Shkolnikov2023}. In contrast we first solve the SPDE and then prove that the associated SDEs are well-posed. Moreover, without environmental noise there is a vast literature on propagation of chaos with non-Lipschitz kernels, see for instance \cite{snitzman_propagation_of_chaos, godinho_Keller_Segel2015, LiuYang2016, Fournier2017, JabinWang2018, Laker2018}. However, most of these works are based on the relative entropy and cannot be simply generalized to a setting with environmental noise, i.e. stochastic framework. For particle systems with environmental noise and non-Lipschitz interaction force~\(k \) (as in our case), to the best of our knowledge there exists no general theory on stochastic McKean-Vlasov equations or on propagation of chaos. In order to derive propagation of chaos for the HK model~\eqref{eq: introduction_mean_field_trajectories_with_common_noise}, we essentially rely on the well-posedness of the McKean--Vlasov SDE~\eqref{eq: introduction_mean_field_trajectories_with_common_noise} and follow \cite{snitzman_propagation_of_chaos} as well as \cite{lazarovici2017mean,chen2021rigorous} to prove that \( (\rho_t , t \ge 0) \) characterizes the measure of the mean-field limit. Initially, all results are formulated in dimension one since this is essential to establish the well-posedness of the McKean--Vlasov equation~\eqref{eq: introduction_mean_field_trajectories_with_common_noise}, see the BSPDE argument in Section~\ref{sec: mean_field_sde} and in particular~\cite[Corollary~2.2]{DuTang2013}. For all other results, we provide corollaries providing the multi-dimensional versions.

\medskip

\noindent \textbf{Organization of the paper:} In Section~\ref{sec:HK model} we introduce the Hegselmann--Krause model with environmental noise and provide necessary definitions and some background information. In Section~\ref{sec:FP equation} the well-posedness of the stochastic Fokker--Planck equation~\eqref{eq: introduction_hk_spde} is established and in Section~\ref{sec: mean_field_sde} of the associated McKean--Vlasov equation~\eqref{eq: introduction_mean_field_trajectories_with_common_noise}. The mean-field limit and propagation of chaos of the HR model with environmental noise are investigated in Section~\ref{sec:mean field limit}.

\section{Hegselmann--Krause model}\label{sec:HK model}

In the following section we introduce the Hegselmann--Krause model and briefly review some of the related literature on the well studied HK model without environmental noise. The Hegselmann--Krause model with environmental noise is set up in Subsection~\ref{subsec: HK model} with its corresponding stochastic Fokker--Planck equation as well as its mean-field equation. We conclude this section with some basic definitions and functions spaces necessary to study the involved equations.

\subsection{Some background information}

In this subsection we present some background on the development of the HK model as well as on the terminology of propagation of chaos. We start with the original discrete-time Hegselsmann--Krause model~\cite{hegselmann2002opinion}, which is given by
\begin{equation}\label{eq:original HK model}
  x_{i}(t+1) = \frac{1}{|\mathscr{N}_{i}(t)|} \sum\limits_{j \in  \mathscr{N}_{i}(t)} x_j(t) ,\quad t \ge 0 , \quad i = 1, \ldots, n ,
\end{equation}
where \(x_{i}(t)\) is the opinion of agent \(i\) at time \(t\), \(\mathscr{N}_{i}(t):= \{ 1 \le j \le n  :  |x_i(t)-x_j(t)| \le r_i \}\) denotes the neighbor set of agent \(i\) at time \(t\) and \(|\mathscr{N}_{i}(t)|\) is the cardinality of the set. The convergence and consensus properties of the discrete-time HK model were extensively  studied in the past years, see for instance \cite{hegselmann2002opinion, lorenz2006consensus, Bhattacharyya2013, 10.1371/journal.pone.0043507,Touri2012}. The main characteristic feature of bounded confidence opinion models, like the HK model~\eqref{eq:original HK model}, is that the agents interact only locally, which is modelled by the compactly supported interaction force in the discrete-time HK model~\eqref{eq:original HK model}, i.e. \(j \in \mathscr{N}_{i}(t)\), and, thus, opinions outside an agent's moral beliefs get ignored through this local interaction kernel. This phenomena is, e.g., observed in case of liberal and conservative view points and their respective social media bubbles in the USA \cite{liberal_conservative,groshek2017helping,fake_news_ideology}. The discrete-time HK model~\eqref{eq:original HK model} is a fairly simple model to describe opinion dynamics and by now there are numerous generalizations and variants of the original HK model, for instance, the HK model with media literacy \cite{media_literacy} or the HK model with an opinion leader \cite{Wongkaew_HK_model_with_leader}. For further extensions we refer to \cite{douven2010extending, riegler2009extending}.

An important class of extensions of the original HK model captures external random effect in opinion dynamics, see e.g. \cite{pineda2013noisy, chen2019heterogeneous}, leading naturally to a system of \(N\) stochastic processes representing the opinion evolution. In this case, following e.g. \cite{garnier2017consensus}, the opinion dynamics $\hat{\mathbf{X}}^N := (\hat{X}^{i}, i = 1,\dots, N)$ of $N$ agents are modelled by a system of stochastic differential equations
\begin{equation}\label{eq: spde_introduction_regularized_particle_system}
  \Id \hat{X}_t^{i} = -\frac{1}{N-1} \sum\limits_{\substack{j=1 \\j \neq i }}^N k_{\scriptscriptstyle{HK}}(\hat{X}_t^{i} -\hat{X}_t^{j}) \Id t + \sigma(t,\hat{X}_t^{i}) \Id B_t^{i}, \quad i=1,\ldots,N ,
  \quad \mathbf{\hat{X}}_0^{N} \sim \overset{N}{\underset{i=1}{\otimes}} \rho_0,
\end{equation}
for $t\geq0$, where \(\hat{X}_t^{i}\) is the \(i\)-th agent's opinion at time \(t\), \(k_{\scriptscriptstyle{HK}}\) is the interaction force between the agents, \(\sigma\) is a smooth diffusion coefficient, \(((B_t^{i}, t \ge 0 ) , i \in \N )\) is a sequence of one-dimensional independent Brownian motion and, as previously, \(\rho_0\) is the initial distribution (of \(\zeta_1\)). We note that the interaction force~$k_{\scriptscriptstyle{HK}}$ has compact support turning the continuous-time HK model~\eqref{eq: spde_introduction_regularized_particle_system} into a bounded confidence model. The continuous-time HK model~\eqref{eq: spde_introduction_regularized_particle_system} has been a topic of active research in the past years, e.g., the convergence to a consensus is studied in \cite{garnier2017consensus} and the phase transition was investigated in \cite{wang2017noisy}.

The particle system~\eqref{eq: spde_introduction_regularized_particle_system} induces in the mean-field limit, i.e. as \(N \to \infty\), a sequence \(((\hat{Y}_t^{i}, t \ge 0 ), i \in \N)\) of i.i.d. stochastic processes satisfying the non-linear, non-local McKean--Vlasov equations
\begin{align} \label{eq: introduction_McKean_Vlasov_for_Hk}
  \begin{cases}
  \Id \hat{Y}_t^{i} = -  (k_{\scriptscriptstyle{HK}} * \hat{\rho_t})(\hat{Y}_t^{i}) \Id t + \sigma(t,\hat{Y}_t^{i} ) \Id B_t^{i},  \quad \hat{Y}_0^{i}=\hat{X}_0^{i},\quad t\geq 0, \\
  \hat{\rho}_t \; \mathrm{is \, the \,  density \, of} \, \hat{Y}_t^{i}.
  \end{cases}
\end{align}
For the specific interaction kernel \(k_{\scriptscriptstyle{HK}} = \indicator{[0,R]}(|x|)x \), each \(\hat{Y}_t^{i}\) corresponds to a well-posed, non-linear Fokker--Planck equation for the agent density profile \(\hat{\rho}(t,x)\) given by
\begin{equation}\label{eq: introduction_Fokker_Planck_for_Hk}
  \partial_t \hat{\rho}(x,t) = \frac{\dd^2 }{\dd x^2} \Big( \frac{\sigma^2}{2} \hat{\rho}_t (x) \Big)+  \frac{\dd }{\dd x}  ((k_{\scriptscriptstyle{HK}}*\hat{\rho}_t)(x) \hat{\rho}_t(x)) ,\quad t\geq 0,
\end{equation}
see \cite{CHAZELLE2017365} for more details.

The given probability measure \(\hat{\rho}(t,x) \Id x \) is classically acquired as the deterministic limit as \(N \to \infty\) of the random measures $\Pi_N$ with values in the space \(\mathcal{P}(C([0,T],\R))\) of probability measures on \(C([0,T],\R)\), defined as
\begin{equation}\label{eq: propagation_of_chaos_empirical_measure}
  \omega \mapsto \Pi_N(\omega,A) := \frac{1}{N} \sum\limits_{i=1}^N \delta_{\hat{X}^{i}_{\cdot}(\omega)}  (A), \quad \quad A \in \mathcal{B}(C([0,T],\R)),
\end{equation}
where \(\delta_f\) is the Dirac measure for \( f \in C([0,T],\R)\). This convergence phenomena is known as propagation of chaos, see e.g. \cite{snitzman_propagation_of_chaos, kac1956foundations}. More precisely, let \(Z = (Z_t, t \ge 0)\) be a continuous \(\R\)-valued stochastic process defined on some probability space such that for each \(t \ge 0\), \(Z_t\) has the law \(\mu_t\). Moreover, let \(\mu_t^N\) be the law of a system \(\mathbf{Z}_t^{N}:= (Z_t^{1}, \ldots, Z_t^{N})\) of $N$ stochastic processes taking values in $\R$. We say that \(\mathbf{Z}_t^N\) is \(\mu_t\)-chaotic, if \(\mu_t^N\) is a symmetric measure and for each fix \(k \in \N\) with \(k \ge 2 \),
\begin{equation}\label{eq: propagation_of_chaos_convergence_condition}
  \mu_t^{k,N} \; \mathrm{converges \; weakly \; to} \; \mu_t \; \mathrm{as}  \; N \to \infty,
\end{equation}
where \(\mu^{k,N}_t(A_1 \times \cdots \times A_k):= \mu^N_t(A_1 \times \cdots \times A_k \times \R \times \cdots \times \R)\) for \(A_1,\ldots, A_k \in \mathcal{B}(\R)\) is the \(k\)-th marginal distribution of \(\mu_t^N\), and we say that propagation of chaos holds if for any time point \(t \ge 0\), \(\mathbf{Z}_t^N\) is \(\mu_t\)-chaotic. It is a classical result that condition~\eqref{eq: propagation_of_chaos_convergence_condition} holds for all \(k\ge 2 \) if and only if condition~\eqref{eq: propagation_of_chaos_convergence_condition} holds for \(k=2\), which is again equivalent to the convergence of the empirical measures \eqref{eq: propagation_of_chaos_empirical_measure} associated to \(\mathbf{Z}\) to the deterministic measure \(\mu\) on \(C([0,T],\R)\), see \cite[Proposition~2.2]{snitzman_propagation_of_chaos}.

It is well-known that if \(k_{\scriptscriptstyle{HK}}\) would satisfy suitable Lipschitz and growth assumptions, the particle system \((\mathbf{\hat{X}}^{N}_t, t \ge 0 )\) satisfies propagation of chaos, see \cite{snitzman_propagation_of_chaos, KURTZ1999}. However, \(k_{\scriptscriptstyle{HK}}\) is not Lipschitz continuous and hence, we can no longer apply the classical theory to obtain the convergence in law of the empirical measure \eqref{eq: propagation_of_chaos_empirical_measure}. Consequently, it requires novel and non-standard methods. An idea, which has been developed for different interacting particle systems without environmental noise in recent years, is to introduce a smooth approximation~\(k_{\scriptscriptstyle{HK}}^N\), which depend on the number of agents~\(N\) and converges in some sense to~\(k_{\scriptscriptstyle{HK}}\). This leads to a regularized version of the system~\eqref{eq: spde_introduction_regularized_particle_system}, which reads as
\begin{equation*}
  \Id \hat{X}_t^{i,N} = -\frac{1}{N-1} \sum\limits_{\substack{j=1 \\j \neq i }}^N k_{\scriptscriptstyle{HK}}^N(\hat{X}_t^{i,N} -\hat{X}_t^{j,N}) \Id t + \sigma(t,\hat{X}_t^{i,N} ) \Id B_t^{i}, \quad i=1,\ldots,N , \quad \mathbf{\hat{X}}_0^{N} \sim \overset{N}{\underset{i=1}{\otimes}} \hat{\rho}_0,
\end{equation*}
for $t\geq 0$, where \(\hat{\mathbf{X}}^{N}_t:= (\hat{X}_t^{1,N}, \ldots, \hat{X}_t^{N,N})\). This allows to introduce an intermediate particle system by
\begin{align*}
  \begin{cases}
  \Id \hat{Y}_t^{i,N} = -(k_{\scriptscriptstyle{HK}}^N * \hat{\rho}_t^N)(\hat{Y}_t^{i,N}) \Id t + \sigma(t, \hat{Y}_t^{i,N} ) \Id B_t^{i}, \quad  \hat{Y}_0^{i,N}=\hat{X}_0^{i,N}, \quad i=1,\ldots,N ,\\
  \hat{\rho}_t \; \mathrm{is \, the \,  density \, of} \, Y_t^{i,N} .
  \end{cases}
\end{align*}
We note that, similar to \eqref{eq: introduction_McKean_Vlasov_for_Hk}, the aforementioned equation is in general a non-linear, non-local McKean--Vlasov SDE, which induces a non-linear Fokker--Planck equation similar to the one presented in~\eqref{eq: introduction_Fokker_Planck_for_Hk}. The regularized systems allow to estimate, for all \(i\), terms of the form \(\E(|\hat{X}_t^{i,N}-\hat{Y}_t^{i,N}| )\) and \(\E(|\hat{Y}_t^{i,N}-\hat{Y}_t^{i}|)\) separately via PDE methods by studying the associated non-linear Fokker--Planck equations. As a result, one can obtain propagation of chaos for the system \eqref{eq: spde_introduction_regularized_particle_system} with \(k_{\scriptscriptstyle{HK}}^N\) instead of \(k_{\scriptscriptstyle{HK}}\). Following the above described method, various versions of propagation of chaos have been shown for a variety of models with general kernels \(k\) in \cite{lazarovici2017mean, canizares2017stochastic, chen2021rigorous} for particle systems without environmental noise, which have non-Lipschitz, unbounded and even singular interaction kernels.

\subsection{Hegselmann--Krause model with environmental noise}\label{subsec: HK model}

In this subsection we introduce the Hegselmann--Krause model with environmental noise, its corresponding mean-field stochastic differential equation with its associated stochastic Fokker--Planck equation.

\medskip

Let \((\Omega, \cF, (\cF_t)_{t \ge 0 } , \P)\) be a complete filtered probability space, \(B=(B_t^{i}, t\ge 0, i\in \N)\) be a sequence of one-dimensional Brownian motions and \(W=(W_t, t \ge 0)\) be a one-dimensional Brownian motion. All Brownian motions $(B_t^{i}, t\ge 0,i\in \N)$ and $(W_t, t \ge 0)$ are supposed to be independent and measurable with respect to the filtration \((\cF_t, t \ge 0)\). Let the initial data \((\zeta^{i} ,i\in \N)\)  be i.i.d. random variables with density~\(\rho_0\) and independent of the Brownian motions \((B_t^{i}, t\ge 0,i\in \N)\), and \((W_t, t \ge 0)\). Moreover,  we denote by \(\cF^W=(\cF^W_t, t \ge 0)\) the augmented filtration generated by \(W\) (see \cite[Section~2.7]{KaratzasIoannis2009Bmas} for the definition) and by \(\mathcal{P}^W\) the predictable \(\sigma\)-algebra with respect to \(\cF^W\). Analogous notation will be used for the filtration generated by the Brownian motion \(B\).

The Hegselmann--Krause model with environmental noise is given by the interacting particle system \(\mathbf{X}^N_t= (X_t^{1}, \ldots, X_t^{N}) \) following the dynamics
\begin{equation}\label{eq: particle_system_with_common_noise}
  \d X_t^{i} = -\frac{1}{N-1} \sum\limits_{\substack{j=1 \\j \neq i }}^N k_{\scriptscriptstyle{HK}}(X_t^{i} -X_t^{j}) \Id t + \sigma(t,X_t^{i}) \Id B_t^{i} + \nu \Id W_t , \quad i=1,\ldots,N , \quad X_0^{i} = \zeta_i,
\end{equation}
for $t\in [0,T]$, where \(\sigma\colon  [0,T] \times \R \mapsto \R \) is the diffusion coefficient, \(\nu > 0\) a constant and the interaction force $k_{\scriptscriptstyle{HK}}(x)=\indicator{[0,R]}(|x|) x$ for $x\in \R$. We point out, that the kernel \(k_{\scriptscriptstyle{HK}}\) always stands for the kernel in the Hegselmann--Krause model. On the other hand \(k\) will denote a general kernel (see Section \ref{sec:FP equation} and Section \ref{sec: mean_field_sde}). 
For establishing propagation of chaos for \(k \in L^1(\R^d) \cap L^2(\R^d)\) we need to introduce an approximation sequence. Let $(k^\tau)_{\tau \in \N}\subset \testfunctions{\R}$ satisfy
\begin{enumerate}[label=(\roman*)]
  \item $\|k^\tau - k\|_{L^2(\R)}\to 0$ as $\tau \to \infty$,
  \item \(\mathrm{supp}(k^\tau)\subset \mathcal{K} \) for $\tau\in \N$ and \(\mathrm{supp}(\frac{\dd}{\dd x} k^\tau )\subset\mathcal{K}\) for some compact set $\mathcal{K}\subset \R$,
  \item \(0 \le k^\tau \le C \), \(|\frac{\dd}{\dd x} k^\tau | \le \frac{C}{\tau} \) for some constant \(C>0 \).
\end{enumerate}

Notice that we by approximating the indicator function by smooth functions \((\psi^\tau, \tau > 0)\)
satisfying
\begin{itemize}
  \item  $\lim\limits_{\tau \to 0} \psi^{\tau}  = \indicator{[-R,R]}$ almost everywhere,
  \item \(\mathrm{supp}(\psi^\tau)\subset [-R-2\tau,R+2\tau ], \), \(\mathrm{supp}(\frac{\dd}{\dd x} \psi^\tau )\subset [-R-2\tau,-R+2\tau ] \cup [R-2\tau,R+2\tau ] \),
  \item \(0 \le \psi^\tau \le 1 \), \(|\frac{\dd}{\dd x} \psi^\tau | \le \frac{C}{\tau} \) for some constant \(C>0 \).
\end{itemize}
we can define
\begin{equation*}
  k_{\scriptscriptstyle{HK}}^\tau(x):= x \psi^\tau(x).
\end{equation*}
Then, \((k_{\scriptscriptstyle{HK}}^\tau, \tau >0)\) satisfies the above approximation properties.
We denote the regularized interacting particle system for the HK kernel \(k_{\scriptscriptstyle{HK}}^\tau\) by \(\mathbf{X}^{N,\tau}_t= (X_t^{1,\tau}, \ldots, X_t^{N,\tau}) \) and it is given by
\begin{equation}\label{eq: regularized_particle_system_with_common_noise}
  \d X_t^{i,\tau} = -\frac{1}{N-1} \sum\limits_{\substack{j=1 \\j \neq i }}^N k_{\scriptscriptstyle{HK}}^\tau(X_t^{i,\tau} -X_t^{j,\tau}) \Id t + \sigma(t, X_t^{i,\tau}) \Id B_t^{i} + \nu \Id W_t  , \; \; X_0^{i,\tau} = \zeta_i,
\end{equation}
for $t\in [0,T]$ and for \(i=1,\ldots,N\). Although the interaction force kernel $k_{\scriptscriptstyle{HK}}$ is non-Lipschitz continuous, the \(N\)-particle systems~\eqref{eq: particle_system_with_common_noise} and \eqref{eq: regularized_particle_system_with_common_noise} possess unique strong solutions, see e.g. \cite[Theorem~1.1]{Pamen2019}, as \(k_{\scriptscriptstyle{HK}}^\tau\) and \(k_{\scriptscriptstyle{HK}}\) are bounded and measurable.

Corresponding to the particle systems \eqref{eq: particle_system_with_common_noise} and \eqref{eq: regularized_particle_system_with_common_noise}, for \(i \in \N\), the system of mean-field SDEs is given by
\begin{align}\label{eq: mean_field_trajectories_with_common_noise}
  \begin{cases}
  \Id Y_t^{i} = -  (k_{\scriptscriptstyle{HK}} * \rho_t)(Y_t^{i}) \Id t + \sigma(t,Y_t^{i}) \Id B_t^{i} + \nu \Id W_t, \quad Y_0^{i}=X_0^{i} , \\
  \rho_t \; \mathrm{is \, the \, conditional \, density \, of} \, Y_t^{i} \, \mathrm{given} \, \cF_t^W,
  \end{cases}
\end{align}
and the system of regularized mean-field SDEs is defined by
\begin{align}\label{eq: regularized_mean_field_trajectories_with_common_noise}
  \begin{cases}
  \Id Y_t^{i,\tau} = -  (k_{\scriptscriptstyle{HK}}^\tau * \rho_t^\tau)(Y_t^{i,\tau}) \Id t + \sigma(t, Y_t^{i,\tau} ) \Id B_t^{i} + \nu \Id W_t, \quad Y_0^{i,\tau}=X_0^{i,\tau}, \\
  \rho_t^\tau \; \mathrm{is \, the \, conditional \, density \, of} \, Y_t^{i,\tau} \, \mathrm{given} \, \cF_t^W ,
  \end{cases}
\end{align}
for $t\in [0,T]$, where \(\rho_t\) denotes the conditional density of \( Y_t^{i}\) given \(\cF^W_t\), that is, for every bounded continuous function \(\varphi\), \(\rho_t\) satisfies
\begin{equation*}
  \E(\varphi(Y_t^{i}) \, | \, \cF^W_t) = \int_{\R} \varphi(x) \rho_t(x) \Id x, \quad \mathbb{P} \text{-a.e.}
\end{equation*}
The same holds for the regularized conditional density \(\rho_t^\tau\) of \(Y^{i, \tau}_t\) given \(\cF^W_t\). Let us remark that \(\rho^\tau_t, \rho_t\) have no superscript \(i\) since they are independent of \(i \in \N \). Indeed, in our case the (regularized) mean-field particles are conditionally independent given \(\cF^W\) and identically distributed, thus, the conditional density is the same for each \(i \in \N\).

Associated to the mean-field SDEs \eqref{eq: mean_field_trajectories_with_common_noise} and \eqref{eq: regularized_mean_field_trajectories_with_common_noise}, the stochastic Fokker--Planck equation reads as
\begin{equation}\label{eq: hk_spde}
  \Id \rho_t = \frac{\dd^2}{\dd x^2} \Big( \frac{\sigma_t^2 + \nu^2 }{2}  \rho_t \Big) \Id t + \frac{\dd}{\dd x} ((k_{\scriptscriptstyle{HK}} *\rho_t)\rho_t) \Id t  - \nu \frac{\dd}{\dd x} \rho_t \Id W_t,
\end{equation}
and the regularized stochastic Fokker--Planck equation as
\begin{equation}\label{eq: regularised_hk_spde}
  \Id \rho_t^\tau = \frac{\dd^2}{\dd x^2} \Big( \frac{\sigma_t^2 + \nu^2 }{2}  \rho_t^\tau  \Big) \Id t + \frac{\dd}{\dd x} ((k_{\scriptscriptstyle{HK}}^\tau  *\rho_t^\tau )\rho_t^\tau ) \Id t - \nu \frac{\dd}{\dd x} \rho_t^\tau  \Id W_t,\quad t\in [0,T].
\end{equation}

Let us remark that we purposely use the same unknown functions \(\rho , \rho^\tau\) for the solutions of the stochastic Fokker--Planck equation~\eqref{eq: hk_spde} and \eqref{eq: regularised_hk_spde} and for the conditional density of the mean-field SDEs \eqref{eq: mean_field_trajectories_with_common_noise} and \eqref{eq: regularized_mean_field_trajectories_with_common_noise} since, as we will see in Theorem~\ref{theorem: existence_mean_field_SDE}, they both coincide under enough regularity assumptions on the initial condition \(\rho_0\). Nevertheless, the meaning of \(\rho , \rho^\tau\) will always be clear from context. We make the following assumptions on the diffusion coefficient~\(\sigma\).

\begin{assumption}\label{ass: sigma_ellipticity+bound}
  Let \(T>0\) and \(\sigma\colon [0,T] \times \R \to \R\) the diffusion coefficient, which satisfies:
  \begin{enumerate}[label=(\roman*)]
    \item There exists a constant \(\lambda > 0\) such that
    \begin{equation*}
      \sigma^2(t,x) \ge \lambda
    \end{equation*}
    for all \(x \in \R\) and \(t \in[ 0,T]\).
    \item There exists a constant \(\Lambda >0\) such that for all \(t \in [0,T]\) we have
    \begin{equation*}
      \sigma(t,\cdot) \in C^3(\R) \quad \mathrm{and} \quad \sup\limits_{t \in [0,T]} \sum\limits_{i=1}^3 \norm{\frac{\dd^{i}}{\dd x^{i}} \sigma(t, \cdot)}_{L^\infty(\R)} \le \Lambda .
    \end{equation*}
  \end{enumerate}
\end{assumption}

The well-posedness of the stochastic Fokker--Planck equation~\eqref{eq: hk_spde} and \eqref{eq: regularised_hk_spde} is presented in Section~\ref{sec:FP equation} and the well-posedness of the mean-field SDEs \eqref{eq: mean_field_trajectories_with_common_noise} and \eqref{eq: regularized_mean_field_trajectories_with_common_noise} in Section~\ref{sec: mean_field_sde}. For this purpose, we first need to fix some basic definitions and function spaces in the next subsection.

\subsection{Function spaces and basic definitions}

In this subsection we collect some basic definitions and introduce the required function spaces. For \( 1 \le p \le \infty\) we denote by \(L^p(\R^d)\) with norm \(\norm{\cdot}_{L^p(\R^d)}\) the vector space of measurable functions whose \(p\)-th power is Lebesgue integrable (with the standard modification for \(p = \infty\)), by \(\testfunctions{\R^d}\) the set of all infinitely differentiable functions with compact support on \(\R^d\) and by \(\mathcal{S}(\R^d)\) the set of all Schwartz functions, see \cite[Chapter~6]{YoshidaKosaku1995FA} for more details. We note that \(\testfunctions{\R^d}\) and \(\mathcal{S}(\R^d)\) are endowed with their standard topologies. Let
\begin{equation*}
  \mathcal{A}:= \{ \alpha = ( \alpha_1,\ldots, \alpha_d) \; : \; \alpha_1, \ldots, \alpha_d \in \N_{0} \}
\end{equation*}
be the set of all multi-indices and \(|\alpha| := \alpha_1 +  \ldots+ \alpha_d\). The derivative will be denoted by
\begin{equation*}
  \partial^\alpha: = \frac{\partial^{|\alpha|}}{\partial x_1^{\alpha_1} \partial x_2^{\alpha_2} \cdots \partial x_d^{\alpha_d} }.
\end{equation*}
In one dimension ($d=1$), we also write \(\frac{\d^n}{\d x^n} f \) for the n-th derivative with respect to \(x \in \R\) of a smooth function~\(f\) defined on~\(\R\). We drop the superscripts \(\alpha,n \) in the case \(\alpha = n = 1 \). Moreover, as an inductive limit, we have \(\testfunctions{\R} = \bigcup\limits_{M=1}^\infty \testfunctions{B(0,M)}\), where \(B(0,M)\) is a ball with radius \(M>0\) in \(\R^d\) and \((\testfunctions{B(0,M)}, p_{\alpha,M} ) \) is the complete metrizable space of smooth functions with compact support in \(B(0,M)\) and semi-norm
\begin{equation*}
  p_{\alpha,M}(f) : =  \sup\limits_{|x| \le M} | (\partial^\alpha f)(x)|
\end{equation*}
for \(f \in \testfunctions{B(0,M)}\) and \(\alpha \in  \mathcal{A}\). We note that this characterization and Baire category theorem immediately imply that \(\testfunctions{\R}\) is not metrizable. Similar, for each \(\alpha, \beta \in \mathcal{A} \) we define the semi-norms
\begin{align*}
  p_{\alpha, \beta} (f):= \sup\limits_{x \in \R^d} | x^\alpha( \partial^\beta f)(x)|.
\end{align*}
Equipped with these seminorms, \(\mathcal{S}(\R^d)\) is  a Fr{\'e}chet space \cite[Appendix~A.5]{AbelsHelmut2012}. Furthermore, we introduce the space of Schwartz distributions \(\mathcal{S}'(\R^d)\). We denote dual parings by \(\qv{\cdot, \cdot}\). For instance, for \(f \in \mathcal{S}', \; u \in \mathcal{S}\) we have \(\qv{u,f}  = u[f]\) and for a probability measure~\(\mu\) we have \(\qv{f,\mu}  = \int f \Id \mu\). The correct interpretation will be clear from the context but should not be confused with scalar product \(\qv{\cdot, \cdot}_{L^2(\R)} \) in \(L^2\).

The Fourier transform \(\mathcal{F}[u]\) and the inverse Fourier transform \(\mathcal{F}^{-1}[u]\) for $u \in \mathcal{S}'(\R^d)$ and $ f \in \mathcal{S}(\R^d)$ are defined by
\begin{equation*}
  \qv{\mathcal{F}[u],f} := \qv{u, \mathcal{F}[f]},
\end{equation*}
where \(\mathcal{F}[f]\) and \(\mathcal{F}^{-1}[f]\) is given by
\begin{equation*}
  \mathcal{F}[f](\xi) := \frac{1}{(2\pi)^{d/2}} \int e^{-i \eta \cdot x } f(x) \Id x
  \quad \text{and}\quad
  \mathcal{F}^{-1}[f](\xi):= \frac{1}{(2\pi)^{d/2}} \int e^{i \eta \cdot x } f(x) \Id x .
\end{equation*}

The Bessel potential for each \(s \in \R\) is denoted by \(J^s := (1-\Delta)^{s/2}u := \mathcal{F}^{-1}[(1+|\xi|^2)^{s/2} \mathcal{F}[u]] \) for \(u \in \mathcal{S}'(\R^d)\). We define the Bessel potential space \(\mathnormal{H}_p^s \) for \(p \in (1,\infty)\) and \(s \in \R\) by
\begin{equation*}
  \mathnormal{H}_p^s:= \{ u \in \mathcal{S}'(\R^d) \; : \;  (1-\Delta)^{s/2} u \in L^p(\R^d) \}
\end{equation*}
with the norm
\begin{equation*}
  \norm{u}_{\mathnormal{H}^{s}_p}:= \norm{(1-\Delta)^{s/2}u}_{L^p(\R^d)}, \quad u \in  \mathnormal{H}_p^s.
\end{equation*}
For \(1 <p < \infty , \; m \in \N\) we can characterize the above Bessel potential spaces \(\mathnormal{H}_p^m\) as Sobolev spaces
\begin{align*}
  W^{m,p} (\R^d) : = \bigg \{ f \in L^p(\R^d) \; : \; \norm{f}_{W^{m,p}(\R^d)}:= \sum\limits_{ \alpha \in \mathcal{A} ,\  |\alpha|\le m } \norm{\partial^\alpha f }_{L^p(\R^d)}  < \infty \bigg \},
\end{align*}
where \(\partial^\alpha f \) is to be understood as weak derivatives \cite{AdamsRobertA2003Ss}. We refer to \cite[Theorem~2.5.6]{TriebelHans1983Tofs} for the proof of the above characterization. As a result, we use Sobolev spaces, which in our context are easier to handle, instead of Bessel potential spaces, whenever possible.

Finally, we introduce general \(L^p\)-spaces, which will serve as the solution space for the SPDEs~\eqref{eq: hk_spde} and \eqref{eq: regularised_hk_spde}. For a Banach space \((Z, \norm{\cdot}_{Z})\), some filtration \((\cF_t)_{t\geq0}\), \(1 \le p \le \infty\) and \(0 \le s <t \le T\) we denote by \(S^p_{\cF}([s,t];Z )\) the set of \(Z\)-valued \((\mathcal{F}_t)\)-adapted continuous processes \((X_u, u \in [s,t])\) such that
\begin{align*}
  \norm{X}_{S^p_{\cF}([s,t];Z )}:=
  \begin{cases}
  \bigg( \E \bigg( \sup\limits_{u \in [s,t]} \norm{X_u}_Z^p\bigg) \bigg)^{\frac{1}{p}}, \quad & p \in [1, \infty)  \\
  \sup\limits_{ \omega \in \Omega } \sup\limits_{u \in [s,t]} \norm{X_u}_Z , \quad & p = \infty
  \end{cases}
\end{align*}
is finite. Similar, \(L^p_{\cF}([s,t];Z)\) denotes the set of \(Z\)-valued predictable processes \((X_u, u \in [s,t])\) such that
\begin{align*}
  \norm{X}_{L^p_{\cF}([s,t];Z )}:=
  \begin{cases}
  \bigg( \E \bigg( \int\limits_{s}^t  \norm{X_u}_Z^p \Id u \bigg) \bigg)^{\frac{1}{p}}, \quad & p \in [1, \infty)  \\
  \sup\limits_{ (\omega,u)  \in \Omega \times [s,t]} \norm{X_u}_Z , \quad & p = \infty
  \end{cases}
\end{align*}
is finite. In most cases \(Z\) will be the Bessel potential space \(\mathnormal{H}_p^n\), as is mainly used by Krylov~\cite{krylov2010ito} in treating SPDEs. For a more detail introduction to the above function spaces we refer to \cite[Section~3]{Krylov1999AnAA}.

\section{Well-posedness of the stochastic Fokker--Planck equations}\label{sec:FP equation}

This section is dedicated to establishing the global existence and uniqueness of weak solutions of the stochastic Fokker--Planck equations~\eqref{eq: hk_spde} and~\eqref{eq: regularised_hk_spde} under suitable conditions on the initial condition and coefficients. Instead of treating the special case \(k_{\scriptscriptstyle{HK}}\), we will take a more general approach and prove existence and uniqueness for general interaction forces \(k \colon \R \to \R\) under some integrability conditions.

Before we start our analysis, we introduce the concept of weak solutions.

\begin{definition}\label{def: solution_of_hk_spde}
  For a general interaction force $k \in L^2(\R)$, a non-negative stochastic process \((\rho_t , t \ge 0)\) is called a (weak) solution of the SPDE
  \begin{equation}\label{eq: hk_spde general}
    \Id \rho_t = \frac{\dd^2}{\dd x^2} \Big( \frac{\sigma_t^2 + \nu^2 }{2}  \rho_t \Big) \Id t + \frac{\dd}{\dd x} ((k *\rho_t)\rho_t) \Id t  - \nu \frac{\dd}{\dd x} \rho_t \Id W_t,\quad t\in [0,T],
  \end{equation}
  if
  \begin{equation*}
    (\rho_t , t \in [0,T])\in  L^2_{\cF^W}([0,T];W^{1,2}(\R) ) \cap S^\infty_{\cF^W}([0,T];L^1(\R) \cap L^2(\R))
  \end{equation*}
  and, for any \(\varphi \in \testfunctions{\R}\), \(\rho\) satisfies almost surely the equation, for all \(t \in [0,T]\),
  \begin{align}\label{eq: solution_hk_spde}
    \begin{split}
    \qv{\rho_t, \varphi}_{L^2(\R)}
    =& \;  \qv{\rho_0, \varphi}_{L^2(\R)}
    +  \int\limits_0^t \bigg\langle \frac{\sigma_s^2 + \nu^2 }{2} \rho_s, \frac{\dd^2}{\dd x^2}\varphi \bigg\rangle_{L^2(\R)} \Id s
    - \int\limits_0^t   \bigg \langle (k *\rho_s)\rho_s, \frac{\dd}{\dd x} \varphi \bigg \rangle_{L^2(\R)}  \Id s  \\
    & +  \int\limits_0^t \nu  \bigg \langle \rho_s, \frac{\dd}{\dd x} \varphi \bigg\rangle_{L^2(\R)}  \Id W_s
    \end{split}
  \end{align}
\end{definition}

\begin{remark}
  A solution to the stochastic partial differential equation~\eqref{eq: hk_spde} and~\eqref{eq: regularised_hk_spde} is defined analogously by replacing \(k\) with \(k_{\scriptscriptstyle{HK}}\) or \(k_{\scriptscriptstyle{HK}}^\tau\), respectively.
\end{remark}

\begin{remark}
  There are multiple solution concepts for SPDEs, see for example \cite{LiuWei2015SPDE} for strong solutions in general separable Hilbert spaces or \cite{DaPratoGiuseppe2014Seii} for mild solutions with respect to a infinitesimal generator. In the present work, we use the concept presented in \cite{Krylov1999AnAA}. This has the advantage that we can use It{\^o}'s formula for \(L^p\)-norms \cite{krylov2010ito} as well as the linear SPDE theory in \cite{Krylov1999AnAA}.
\end{remark}

\begin{remark}
  Under the assumption that for all \(t \in [0,T]\), \(\sigma_t \in C^2(\R)\) we can rewrite formally equation~\eqref{eq: solution_hk_spde} such that the leading coefficient is in non-divergence form, i.e.
  \begin{align*}
    &\qv{\rho_t, \varphi}_{L^2(\R)}\\
    &\quad= \;  \qv{\rho_0, \varphi}_{L^2(\R)}
    + \frac{1}{2}  \int\limits_0^t \bigg \langle (\sigma_s^2 + \nu^2 ) \frac{\dd^2}{\dd x^2} \rho_s, \varphi \bigg \rangle_{L^2(\R)}
    + 2  \bigg\langle  \frac{\dd}{\dd x}(\sigma_s^2 + \nu^2 ) \frac{\dd}{\dd x} \rho_s, \varphi \bigg\rangle_{L^2(\R)} \\
    & \quad\quad+  \bigg\langle \frac{\dd^2}{\dd x^2}(\sigma_s^2 + \nu^2 ) \rho_s, \varphi  \bigg\rangle_{L^2(\R)} \Id s
    -\int\limits_0^t \bigg \langle (k *\rho_s)\rho_s, \frac{\dd}{\dd x} \varphi\bigg \rangle_{L^2(\R)} \Id s \\
    & \quad\quad -  \int\limits_0^t \nu  \bigg\langle \frac{\dd}{\dd x} \rho_s , \varphi \bigg\rangle _{L^2(\R)} \Id W_s .
  \end{align*}
  Hence, \((\rho_t, t \ge 0)\) solves the following SPDE
  \begin{align*}
    \Id \rho_t =&\;  \frac{\sigma_t^2    + \nu^2 }{2} \frac{\dd^2}{\dd x^2} \rho_t  \Id t +  \frac{\dd}{\dd x}    (\sigma_t^2 + \nu^2)   \frac{\dd}{\dd x} \rho_t  \Id t \\
    &+ \frac{1}{2} \frac{\dd^2}{\dd x^2}(\sigma_t^2 + \nu^2 )  \rho_t  \Id t
    +  \frac{\dd}{\dd x} ((k *\rho_t)\rho_t) \Id t - \nu \frac{\dd}{\dd x}  \rho_t \Id W_t , \quad t\in [0,T].
  \end{align*}
\end{remark}

In the next theorem we establish uniqueness and local existence of weak solutions to the non-local stochastic Fokker--Planck equation~\eqref{eq: hk_spde general}. Furthermore, we are going to see in Corollary~\ref{cor: spde_global_weak_solution} that the existence will not depend on the \(L^2\)-norm of the initial condition~\(\rho_0\), allowing us to extend the local solution obtained in Theorem~\ref{theorem: existence_of_hk_spde} to a global solution on an arbitrary interval~\([0,T ]\).

\begin{theorem}\label{theorem: existence_of_hk_spde}
  Let Assumption~\ref{ass: sigma_ellipticity+bound} hold. Further, assume \(0 \le \rho_0 \in L^1(\R) \cap L^2(\R) \) with \(\norm{\rho_0}_{L^1(\R)}=1\) and $k \in L^2(\R)$. Then, there exists a \(T^*>0\) and a unique non-negative solution of the SPDE~\eqref{eq: hk_spde general} in the space
  \begin{equation*}
    \mathbb{B} : =  L^2_{\cF^W}([0,T^*];W^{1,2}(\R) ) \cap S^\infty_{\cF^W}([0,T^*];L^1(\R) \cap L^2(\R)). 
  \end{equation*}
  Moreover, the solution \(\rho\) has the property of mass conservation 
  \begin{equation*}
    \norm{\rho_t}_{L^1(\R)} = 1 , \quad \P\text{-a.s.},
  \end{equation*}
  for all \(t \in [0,T]\). 
\end{theorem}

\begin{proof}
  Let us define the metric space
  \begin{equation*}
    F^{T,M}:=\big \{X \in S^\infty_{\cF^W}([0,T];L^2(\R) )\; : \; \norm{X}_{S^\infty_{\cF^W}([0,T];L^2(\R) )}  \le M \big\}
  \end{equation*}
  for some constant \(M> \norm{\rho_{0}}_{L^2(\R)} \), for instance \(M= 2 \norm{\rho_{0}}_{L^2(\R)}\). The metric on~$F^{T,M}$ is induced by the norm on  $S^\infty_{\cF^W}([0,T];L^2(\R) )$. The solution map \(\mathcal{T} \colon F^{T,M} \to F^{T,M}\) is defined as follows. For each \(\zeta \in F^{T,M}\) we define \(\mathcal{T}(\zeta)\) as the solution of the following linear SPDE
  \begin{align}\label{eq; linear_hk_spde}
    \begin{split}
    \Id \rho_t =&\;  \frac{\sigma_t^2 + \nu^2 }{2} \frac{\dd^2}{\dd x^2} \rho_t   \Id t +   \frac{\dd}{\dd x}  (\sigma_t^2 + \nu^2)   \frac{\dd}{\dd x} \rho_t  \Id t \\
    &+ \frac{1}{2} \frac{\dd^2}{\dd x^2}(\sigma_t^2 + \nu^2 )  \rho_t  \Id t
    +  \frac{\dd}{\dd x} ((k *\zeta_t)\rho_t) \Id t - \nu \frac{\dd}{\dd x}  \rho_t \Id W_t, \quad t\in [0,T].
    \end{split}
  \end{align}
  The \(L^2\)-bound on \(k\) and H{\"o}lder's inequality imply
  \begin{equation}\label{eq: l_infinity_bound_k_zeta}
    |k *\zeta_t(x)| \le \norm{k}_{L^2(\R)} \norm{\zeta_t}_{L^2(\R)} \le  \norm{k}_{L^2(\R)} M
  \end{equation}
  for all \(x \in \R\), which allows us to check the conditions of the \(L^p\)-theory of SPDEs \cite[Theorem~5.1 and Theorem~7.1]{Krylov1999AnAA} for the case \(n=-1\) therein. For instance, if we define  for \(q \in W^{1,2}(\R)\) the function
  \begin{equation*}
    f(q,t,x) = \frac{\dd}{\dd x} ((k*\zeta_t) q_t ),
  \end{equation*}
  then obviously \(f(0,\cdot,\cdot) \in L^2_{\cF^W}([0,T]; H^{-1,2}(\R))\) and, since \(x {/} (1+|x|^2)^{1/2} \) is bounded (see \cite[Theorem~2.3.8]{Triebel1978} for the lifting property), we have
  \begin{equation*}
    \norm{f(q,t,\cdot)}_{H^{-1,2} (\R)} \le C \norm{(k*\zeta_t) q_t}_{L^2(\R)}
    \le \norm{k}_{L^2(\R)}  M  \norm{q_t}_{L^2(\R)} .
  \end{equation*}
  By \cite[Remark~5.5]{Krylov1999AnAA} this is sufficient to verify \cite[Assumption~5.6]{Krylov1999AnAA}. The other assumptions are proven similarly.

  Hence, we can deduce that the linear SPDE~\eqref{eq; linear_hk_spde} admits a unique solution
  \begin{equation*}
    \rho^\zeta \in L^2_{\cF^W}([0,T];W^{1,2}(\R) ) \cap S^2_{\cF^W}([0,T];L^2(\R) ).
  \end{equation*}
  In the next step we want to demonstrate the non-negativity of the solution \(\rho^\zeta\) with the regularity of the solution \(\rho^\zeta\). Let us denote by \(k_m\) the mollification of \(k\) and let
  \begin{equation*}
    \rho^{\zeta, m} \in L^2_{\cF^W}([0,T];W^{1,2}(\R) ) \cap S^2_{\cF^W}([0,T];L^2(\R) )
  \end{equation*}
  be the solution of the SPDE~\eqref{eq; linear_hk_spde} with \((k_m*\zeta) \rho^\zeta\) instead of \((k*\zeta) \rho^\zeta\). Then we can write the SPDE in the form
  \begin{align*}
    \Id  \rho^{\zeta, m}_t =&\; a(t,x) \frac{\dd^2}{\dd x^2} \rho^{\zeta, m}_t \Id t  + b^m(t,x)  \frac{\dd}{\dd x}  \rho^{\zeta, m}_t \Id t
    + c^m(t,x) \rho^{\zeta, m}_t  \Id t
    -  \nu \frac{\dd}{\dd x}  \rho^{\zeta, m}_t  \Id W_t ,
  \end{align*}
  for $t \in [0,T]$, with
  \begin{equation*}
    a(t,x):= \frac{\sigma_t^2 + \nu^2 }{2},\,\,
    b^m(t,x):= \frac{\dd}{\dd x}  (\sigma_t^2 + \nu^2)  +   k_m *\zeta_t, \,\,
    c^m(t,x):=\frac{1}{2} \frac{\dd^2}{\dd x^2}(\sigma_t^2 + \nu^2 )  + \frac{\dd}{\dd x} k_m *\zeta_t.
  \end{equation*}
  Now, by Assumption~\ref{ass: sigma_ellipticity+bound} the coefficients \(a^m,b^m,c^m\) and the coefficient in the stochastic part is bounded. Hence, by the maximum principle \cite[Theorem~5.12]{Krylov1999AnAA} the solution \(\rho^{\zeta, m}\) is non-negative. On the other hand, we have
  \begin{align*}
    &\norm{\frac{\dd }{\dd x} \big( (k_m*\zeta_t) \rho^\zeta_t -  (k*\zeta_t)  \rho^\zeta_t \big)}_{ L^2_{\cF^W}([0,T]; H^{-1,2}(\R))}^2\\
    &\quad \le C \norm{((k_m-k)*\zeta_t)  \rho^\zeta_t}_{ L^2_{\cF^W}([0,T]; L^2(\R))}^2 \\
    &\quad \le C \E\bigg( \int\limits_0^T     \norm{((k_m-k)*\zeta_t)}_{L^\infty(\R)}^2    \norm{\rho^\zeta_t}_{L^2(\R)}^2  \bigg)   \\
    &\quad \le C \E\bigg( \int\limits_0^T     \norm{((k_m-k)}_{L^2(\R)}^2 \norm{\zeta_t}_{L^2(\R)}^2 \norm{\rho^\zeta_t}_{L^2(\R)}^2 \bigg)  \\
    &\quad \le C \norm{((k_m-k)}_{L^2(\R)}^2  M^2    \norm{\rho^\zeta_t}_{L^2_{\cF^W}([0,T];L^2(\R))}^2 \xrightarrow[]{m \to \infty} 0.
  \end{align*}
  Consequently, by \cite[Theorem~5.7]{Krylov1999AnAA} we have
  \begin{equation*}
    \lim\limits_{m \to \infty} \norm{\rho^{\zeta,m} - \rho^\zeta }_{ L^2_{\cF^W}([0,T]; W^{1,2}(\R))}  = 0
  \end{equation*}
  and therefore \(\rho^\zeta_t(\cdot)\ge 0 \) for all \(t \in [0,T]\) almost surely (by intersecting all sets of measure one, where \(\rho^{\zeta,m}\) is non-negative).
  
  The non-negativity of the solution \(\rho^\zeta\) and the divergence structure of the equation provides us with the normalization condition/mass conservation, that is
  \begin{equation*}
    \norm{\rho_t^\zeta}_{L^1(\R)} = \norm{\rho_0}_{L^1(\R)} = 1, \quad \P\text{-a.s.},
  \end{equation*}
  for \(t \in [0,T]\). This follows immediately by plugging in a cut-off sequence \((\xi_n, n \in \N)\) for our test function \(\varphi\) and taking the limit \(n \to \infty\) (see \cite[p.~212]{BrezisHaim2011FaSs} for properties of the cut-off sequence). Therefore, the map \(\mathcal{T}(\zeta) = \rho^\zeta\) will be well-defined if we can obtain a bound on the \(S^\infty_{\cF^W}([0,T];L^2(\R) ) \)-norm. For readability we will from now on drop the superscript \(\zeta\) in the following. Applying It{\^o}'s formula~\cite{krylov2010ito}, we obtain
  \begin{align*}
    & \norm{\rho_t}_{L^2(\R)}^2 - \norm{\rho_0}_{L^2(\R)}^2 \\
    &\quad= 2 \nu  \int\limits_0^t  \bigg \langle \rho_s ,  \frac{\dd}{\dd x} \rho_s \bigg \rangle_{L^2(\R)} \Id W_s + \int\limits_0^t  \bigg \langle \rho_s, \frac{\dd}{\dd x}( \sigma_s^2+\nu^2) \frac{\dd}{\dd x} \rho_s + \rho_s \frac{\dd^2}{\dd x^2} ( \sigma_s^2+\nu^2) \bigg \rangle_{L^2(\R)} \Id s \\
    &\quad\quad-  \int\limits_0^t  \bigg\langle (\sigma_s^2+\nu^2) \frac{\dd}{\dd x} \rho_s ,  \frac{\dd}{\dd x}\rho_s \bigg \rangle_{L^2(\R)}   \Id s
    -  2 \int\limits_0^t  \bigg \langle (k *\zeta_s)\rho_s , \frac{\dd}{\dd x}\rho_s \bigg \rangle_{L^2(\R)}   \Id s\\
    &\quad\quad+  \nu^2  \int\limits_0^t  \norm{ \frac{\dd}{\dd x} \rho_s}_{L^2(\R)}^2   \Id s   \\
    &\quad= \int\limits_0^t  \bigg \langle \rho_s,\frac{\dd}{\dd x} ( \sigma_s^2+\nu^2) \frac{\dd}{\dd x} \rho_s + \rho_s \frac{\dd^2}{\dd x^2} ( \sigma_s^2+\nu^2) \bigg \rangle_{L^2(\R)} \Id s -  2 \int\limits_0^t  \bigg \langle (k *\zeta_s)\rho_s ,  \frac{\dd}{\dd x} \rho_s \bigg\rangle_{L^2(\R)}   \Id s \\
    &\quad\quad-  \int\limits_0^t  \bigg \langle \sigma^2_s \frac{\dd}{\dd x} \rho_s , \frac{\dd}{\dd x} \rho_s \bigg \rangle_{L^2(\R)}   \Id s \\
    &\quad  \le \int\limits_0^t  \bigg \langle \rho_s, \frac{\dd}{\dd x}( \sigma_s^2+\nu^2) \frac{\dd}{\dd x} \rho_s + \rho_s \frac{\dd^2}{\dd x^2} ( \sigma_s^2+\nu^2) \bigg \rangle_{L^2(\R)} \Id s -  2 \int\limits_0^t  \bigg \langle (k *\zeta_s)\rho_s ,  \frac{\dd}{\dd x} \rho_s\bigg \rangle_{L^2(\R)}   \Id s \\
    &\quad\quad- \lambda  \int\limits_0^t  \norm{\frac{\dd}{\dd x} \rho_s }_{L^2(\R)}^2   \Id s,
  \end{align*}
  for \(0 \le t \le T\), where we used the fact that \(\rho_s   \frac{\dd}{\dd x} \rho_s = \frac{1}{2} \frac{\dd}{\dd x} (\rho_s^2) \) to get rid of the stochastic integral. At this step, it is crucial that we only have additive common noise. Otherwise the stochastic integral will not vanish and the above estimate will not achieve the \(L^\infty\)-bound in~\(\omega\). For the first term we can use Assumption~\ref{ass: sigma_ellipticity+bound} and Young's inequality to find
  \begin{align}\label{eq: hk_existence_extra_term_sigma_young}
    &\bigg|\int\limits_0^t  \bigg \langle \rho_s, \frac{\dd}{\dd x}( \sigma_s^2+\nu^2) \frac{\dd}{\dd x} \rho_s + \rho_s \frac{\dd^2}{\dd x^2} ( \sigma_s^2+\nu^2) \bigg\rangle_{L^2(\R)} \Id s \bigg|\\
    & \quad\le   \Lambda \int\limits_0^t  \bigg|\bigg \langle \rho_s, \frac{\dd}{\dd x} \rho_s + \rho_s  \bigg \rangle_{L^2(\R)} \bigg| \Id s \nonumber\\
    & \quad\le   \frac{\lambda}{4} \int\limits_0^t \norm{\frac{\dd}{\dd x} \rho_s}_{L^2(\R)}^2  \Id s + \Big( \frac{ \Lambda ^2}{\lambda} + \Lambda \Big)  \int\limits_0^t \norm{ \rho_s}_{L^2(\R)}^2  \Id s.\nonumber
  \end{align}
  On the other hand, using \eqref{eq: l_infinity_bound_k_zeta} and Young's inequality, we obtain
  \begin{align*}
    \bigg|\bigg \langle (k *\zeta_s)\rho_s ,  \frac{\dd}{\dd x} \rho_s \bigg \rangle_{L^2(\R)}\bigg|
    &\le \norm{k *\zeta_s}_{L^\infty(\R)} \bigg \langle |\rho_s|, \bigg| \frac{\dd}{\dd x} \rho_s \bigg| \bigg \rangle_{L^2(\R)} \\
    &\le \norm{k}_{L^2(\R)} M  \norm{\rho_s}_{L^2(\R)}  \norm{ \frac{\dd}{\dd x} \rho_s}_{L^2(\R)}  \\
    &\le \frac{\norm{k}_{L^2(\R)}^2 M^2 }{\lambda} \norm{\rho_s}_{L^2(\R)}^2 + \frac{\lambda}{4} \norm{ \frac{\dd}{\dd x} \rho_s}_{L^2(\R)}^2.
  \end{align*}
  After absorbing the terms, we find 
  \begin{equation*}
    \norm{\rho_t}_{L^2(\R)}^2 - \norm{\rho_0}_{L^2(\R)}^2
    \le \bigg( \frac{\norm{k}_{L^2(\R)}^2 M^2 + \Lambda^2 }{\lambda} + \Lambda \bigg)\int\limits_0^t  \norm{\rho_s}_{L^2(\R)}^2 \Id  s.
  \end{equation*}
  For the rest of the proof we define the constant 
  \begin{equation*}
    C(\lambda, \Lambda, k,M):=  \frac{\norm{k}_{L^2(\R)}^2 M^2 + \Lambda^2 }{\lambda} + \Lambda
  \end{equation*}
  and conclude
  \begin{equation}\label{eq: l_infinity_l_2_bound_of_rho}
    \norm{\rho_t}_{L^2(\R)}^2 \le  \norm{\rho_0}_{L^2(\R)}^2  \exp \big(    C(\lambda, \Lambda, k,M) T  \big) ,
  \end{equation}
  by Gronwall's inequality. Choosing \(\hat{T}^*< \ln(M/ \norm{\rho_0}_{L^2(\R)}^2)    C(\lambda, \Lambda, k,M)^{-1} \), we have \(\rho \in F^{\hat{T}^*,M} \) and the map
  \begin{equation*}
    \mathcal{T}\colon F^{\hat{T}^*,M}  \to F^{\hat{T}^*,M} ,  \quad \zeta \to \rho^\zeta,
  \end{equation*}
  is well-defined up to time \(\hat{T}^*\).

  The next step is to show that \(\mathcal{T}\) is a contraction in a small time span (\(T\le \hat{T}^*\)) and, therefore, has a fixed point. For \(\zeta, \tilde{\zeta} \in F^{T,M} \) let \(\rho:=\mathcal{T}(\zeta), \tilde{\rho}:=\mathcal{T}(\tilde{\zeta})\) be the associated solutions of the linear SPDE~\eqref{eq; linear_hk_spde}. Then, we have
  \begin{align*}
    \Id (\rho_t -\tilde{\rho}_t) =&\;  \frac{\sigma_t^2 + \nu^2 }{2} \frac{\dd^2}{\dd x^2} (\rho_t -\tilde{\rho}_t) \Id t  +  (\sigma_t^2 + \nu^2)   \frac{\dd}{\dd x} (\rho_t -\tilde{\rho}_t)
    \Id t  + \frac{1}{2} \frac{\dd^2}{\dd x^2}(\sigma_t^2 + \nu^2 )  (\rho_t -\tilde{\rho}_t)  \Id t \\
    &+  \frac{\dd}{\dd x} ((k *\zeta_t)\rho_t) \Id t - \frac{\dd}{\dd x} ((k *\tilde{\zeta_t})\tilde{\rho}_t) \Id t -  \nu \frac{\dd}{\dd x}  (\rho_t -\tilde{\rho}_t) \Id W_t, \quad t\in [0,T].
  \end{align*}
  Applying It{\^o}'s formula \cite{krylov2010ito} and multiple Young's inequality again (see~\eqref{eq: hk_existence_extra_term_sigma_young}), we obtain
  \begin{align*}
    &\norm{\rho_t-\tilde{\rho}_t}_{L^2(\R)}^2  \\
    &\quad= - \int\limits_0^t   \bigg \langle (\sigma_s^2 +\nu^2) \frac{\dd}{\dd x} \rho_s - \frac{\dd}{\dd x} \tilde{\rho}_s,  \frac{\dd}{\dd x} \rho_s -  \frac{\dd}{\dd x} \tilde{\rho}_s \bigg \rangle_{L^2(\R)}   \Id s\\
    &\quad\quad -  2 \int\limits_0^t  \bigg \langle (k *\zeta_s)\rho_s - (k *\tilde{\zeta_s}) \tilde{\rho}_s  ,  \frac{\dd}{\dd x} \rho_s - \frac{\dd}{\dd x} \tilde{\rho}_s \bigg \rangle_{L^2(\R)}   \Id s
    +  \nu^2  \int\limits_0^t  \norm{\frac{\dd}{\dd x} (\rho_s-\tilde{\rho}_s)}_{L^2(\R)}^2   \Id s\\
    &\quad\quad + \int\limits_0^t  \bigg \langle \rho_s-\tilde{\rho}_s, \frac{\dd}{\dd x} ( \sigma_s^2+\nu^2) \bigg( \frac{\dd}{\dd x} \rho_s - \frac{\dd}{\dd x} \tilde{\rho}_s \bigg) + (\rho_s-\tilde{\rho}_s) \frac{\dd^2}{\dd x^2} ( \sigma_s^2+\nu^2) \bigg \rangle_{L^2(\R)} \Id s \\
    &\quad\le - \lambda \int\limits_0^t  \norm{ \frac{\dd}{\dd x} \rho_s - \frac{\dd}{\dd x}\tilde{\rho}_s}_{L^2(\R)}^2   \Id s \\
    &\quad\quad- 2 \int\limits_0^t  \bigg \langle (k *(\zeta_s - \tilde{\zeta}_s)) \rho_s + (k *\tilde{\zeta_s}) (\rho_s-\tilde{\rho}_s)  ,  \frac{\dd}{\dd x} \rho_s - \frac{\dd}{\dd x} \tilde{\rho}_s \bigg \rangle_{L^2(\R)}   \Id s  \\
    &\quad \quad+  \frac{\lambda}{4} \int\limits_0^t \norm{\frac{\dd}{\dd x} \rho_s - \frac{\dd}{\dd x} \tilde{\rho}_s}_{L^2(\R)}^2  \Id s + \Big( \frac{ \Lambda ^2}{\lambda} + \Lambda \Big)  \int\limits_0^t \norm{ \rho_s - \tilde{\rho}_s}_{L^2(\R)}^2  \Id  s \\
    &\quad\le  - \frac{3\lambda}{4}\int\limits_0^t   \norm{\frac{\dd}{\dd x}  \rho_s- \frac{\dd}{\dd x}\tilde{\rho}_s}_{L^2(\R)}^2    \Id s
     + \Big( \frac{ \Lambda ^2}{\lambda} + \Lambda \Big)  \int\limits_0^t \norm{ \rho_s - \tilde{\rho}_s}_{L^2(\R)}^2  \Id  s\\
    &\quad\quad+  \int\limits_0^t  \norm{k}_{L^2(\R)} \norm{\zeta_s-\tilde{\zeta}_s}_{L^2(\R)} \norm{\rho_s}_{L^2(\R)} \norm{ \frac{\dd}{\dd x} \rho_s - \frac{\dd}{\dd x} \tilde{\rho}_s}_{L^2(\R)} \Id s  \\
    & \quad\quad+  \norm{k}_{L^2(\R)}  \int\limits_0^t  \norm{\rho_s-\tilde{\rho}_s}_{L^2(\R)} \norm{\tilde{\zeta}_s}_{L^2(\R)} \norm{ \frac{\dd}{\dd x}\rho_s - \frac{\dd}{\dd x}\tilde{\rho}_s}_{L^2(\R)} \Id s \\
    &\quad\le   - \frac{3\lambda}{4}\int\limits_0^t  \norm{\frac{\dd}{\dd x}  \rho_s - \frac{\dd}{\dd x}  \tilde{\rho}_s}_{L^2(\R)}^2   \Id s
    +  \Big( \frac{ \Lambda ^2}{\lambda} + \Lambda \Big)  \int\limits_0^t \norm{ \rho_s - \tilde{\rho}_s}_{L^2(\R)}^2  \Id  s \\
    & \quad\quad+  \int\limits_0^t \frac{\norm{k}_{L^2(\R)}^2 M^2}{\lambda} \norm{\zeta_s-\tilde{\zeta}_s}_{L^2(\R)}^2 +  \frac{\lambda}{4}\norm{\frac{\dd}{\dd x}  \rho_s -  \frac{\dd}{\dd x}  \tilde{\rho}_s}_{L^2(\R)}^2 \Id s  \\
    &\quad\quad +  \int\limits_0^t  \frac{\norm{k}_{L^2(\R) }^2 M^2}{\lambda} \norm{\rho_s-\tilde{\rho}_s}_{L^2(\R)}^2
    + \frac{\lambda}{4} \norm{\frac{\dd}{\dd x}\rho_s - \frac{\dd}{\dd x} \tilde{\rho}_s}_{L^2(\R)}^2 \Id s \\
    &\quad\le  \int\limits_0^t \frac{\norm{k}_{L^2(\R)}^2 M^2}{\lambda} \norm{\zeta_s-\tilde{\zeta}_s}_{L^2(\R)}^2 \Id s +
    \int\limits_0^t   C(\lambda, \Lambda, k,M) \norm{\rho_s-\tilde{\rho}_s}_{L^2(\R)}^2 \Id s \\
    &\quad\le   T  \frac{\norm{k}_{L^2(\R)}^2 M^2}{\lambda} \norm{\zeta-\tilde{\zeta}}_{S_{\cF^W}^\infty([0,T];L^2(\R))}^2
    + C(\lambda, \Lambda, k,M) \int\limits_0^t   \norm{\rho_s-\tilde{\rho}_s}_{L^2(\R)}^2 \Id s .
  \end{align*}
  Gronwall's inequality provide us with the estimate
  \begin{equation*}
    \|\rho-\tilde{\rho}\|_{S_{\cF^W}^\infty([0,T];L^2(\R))}
    \le \sqrt{  \frac{ T \norm{k}_{L^2(\R)}^2 M^2}{\lambda} } \exp \Bigg( \frac{T}{2}   C(\lambda, \Lambda, k,M)\Bigg)
    \|\zeta-\tilde{\zeta}\|_{S_{\cF^W}^\infty([0,T];L^2(\R))} .
  \end{equation*}
  Now, choosing \(T^{*} \) such that
  \begin{equation*}
    \sqrt{\frac{ T^{*} \norm{k}_{L^2(\R)}^2 M^2}{\lambda} {\lambda} } \exp \Bigg( \frac{T^{*}}{2}   C(\lambda, \Lambda, k,M)\Bigg) < 1
  \end{equation*}
  and \(T^{*} \le \hat{T}^{*}\), we see that the map \(\mathcal{T} \colon F^{T^{*},M} \to F^{T^{*},M} \) is a contraction and consequently we obtain a fixed point \(\rho\), which is a local weak solution up to the time \(T^{*}\) of the SPDE~\eqref{eq: hk_spde general}.
\end{proof}

We notice that in the proof of Theorem~\ref{theorem: existence_of_hk_spde}, we only use H{\"o}lder's inequality, Young's convolution and product inequality. Hence, the statement of Theorem~\ref{theorem: existence_of_hk_spde} holds also for arbitrary dimension. We state this observation in the following corollary.

\begin{corollary}
  Assume \(0 \le \rho_0 \in L^1(\R^d) \cap L^2(\R^d) \) with \(\norm{\rho_0}_{L^1(\R^d)}=1\) and consider a general interaction force \(k \in L^2(\R^d)\). Then, there exists a \(T^*>0\) and a unique non-negative solution of the SPDE~\eqref{eq: hk_spde general} in the space
  \begin{equation*}
    \mathbb{B} : =  L^2_{\cF^W}([0,T^*];W^{1,2}(\R^d) ) \cap S^\infty_{\cF^W}([0,T^*];L^1(\R^d) \cap L^2(\R^d)). 
  \end{equation*}
  The solution should be understood in the sense of Definition~\ref{def: solution_of_hk_spde}, where Definition~\ref{def: solution_of_hk_spde} is modified for arbitrary dimension \(d\) in the obvious way, see also \cite[Definition~3.5]{Krylov1999AnAA}.
\end{corollary}

\begin{remark}\label{remark: spde_global_solution_weak_existence}
  Following the steps of the proof of Theorem~\ref{theorem: existence_of_hk_spde}, we see that we can obtain not only a local solution but a (global) solution for any \(T>0\) by requiring a small \(L^2\)-norm on the initial condition \(\rho_0\). In particular, we can choose a constant \(M >0\) such that
  \begin{equation*}
    \sqrt{\frac{ T \norm{k}_{L^2(\R)}^2 M^2}{\lambda} {\lambda} } \exp  \Bigg ( \frac{T}{2}   \Bigg( \frac{\norm{k}_{L^2(\R)}^2 M^2 + \Lambda^2 }{\lambda} + \Lambda \Bigg) \Bigg) < 1
  \end{equation*}
  and then the condition
  \begin{equation*}
    \norm{\rho_0}_{L^2(\R)}\le M \exp\Bigg(- T \Bigg( \frac{\norm{k}_{L^2(\R)}^2 M^2 + \Lambda^2 }{\lambda} + \Lambda \Bigg)  \Bigg)
  \end{equation*}
  guarantees a unique non-negative solution of the SPDE~\eqref{eq: hk_spde general} on the interval \([0,T]\).
\end{remark}

Next, we establish another global existence and uniqueness result. We emphasize that in the following result we do not need any further assumptions on \(\rho_0\) besides being in \(L^1(\R) \cap L^2(\R)\). Instead, we impose a lower bound on the diffusion coefficient~\(\sigma\). Hence, we require a sufficiently high randomness in the stochastic Fokker--Planck equation. We also assert the fact that the continuation of the solution \((\rho_t, t \ge 0)\) is a direct consequence of the \(L^2\)-theory of SPDEs.

\begin{corollary}\label{cor: spde_global_weak_solution}
  Let Assumption~\ref{ass: sigma_ellipticity+bound} hold. Further, assume \(0 \le \rho_0 \in L^1(\R) \cap L^2(\R) \) with \(\norm{\rho_0}_{L^1(\R)}=1\) and $k\in L^1(\R) \cap L^2(\R)$. Furthermore, assume that the diffusion coefficient \(\sigma \) has a derivative \(\frac{\dd }{\dd x} \sigma\) with compact support \([-L,L]\) and satisfies
  \begin{equation}\label{eq: L_2_condition}
       2 L^2 \sup\limits_{0 \le t \le T} \norm{\frac{\dd}{\dd x}  (\sigma^2_t)}_{L^\infty(\R)}
    + \bigg(C_{GNS}^4 \norm{k}_{L^1(\R)}^2 + 4 L^4 \sup\limits_{0 \le t \le T} \norm{\frac{\dd }{\dd x} (\sigma^2_t)}_{L^\infty(\R)}^2\bigg)^{1/2} \le \lambda,
  \end{equation}
  where \(\lambda\) is the ellipticity constant in Assumption~\ref{ass: sigma_ellipticity+bound} \(C_{GNS}\) is the constant given by the Gagliardo--Nirenberg--Sobolev interpolation in one dimension \cite[Theorem~12.83]{LeoniGiovanni2017Afci}, i.e. \(C_{GNS} = \big( \frac{4 \pi^2}{9} \big)^{-1/4}\). Then, for each \(T>0\) there exist unique global non-negative solutions of the stochastic Fokker--Planck equations ~\eqref{eq: hk_spde general} in the space \(\mathbb{B} \).
\end{corollary}

\begin{proof}
  Let \(\rho \) be the solution given by Theorem~\ref{theorem: existence_of_hk_spde}. Following the proof of Theorem~\ref{theorem: existence_of_hk_spde}, we apply It{\^o}'s formula \cite{krylov2010ito} and obtain, for \(0 \le t \le T^*\),
  \begin{align*}
    &\norm{\rho_t}_{L^2(\R)}^2 - \norm{\rho_0}_{L^2(\R)}^2 \\
    &\quad= -\int\limits_0^t  \norm{  \sigma_s \frac{\dd}{\dd x} \rho_s }_{L^2(\R)}^2   \Id s
    +\int\limits_0^t  \bigg\langle \rho_s, \frac{\dd}{\dd x} ( \sigma_s^2+\nu^2) \frac{\dd}{\dd x} \rho_s + \rho_s \frac{\dd^2}{\dd x^2} ( \sigma_s^2+\nu^2) \bigg \rangle \Id s\\
    &\quad\quad-  2 \int\limits_0^t  \bigg\langle (k *\rho_s)\rho_s , \frac{\dd}{\dd x}  \rho_s \bigg \rangle_{L^2(\R)}   \Id s  \\
    &\quad= -\int\limits_0^t  \norm{  \sigma_s \frac{\dd}{\dd x} \rho_s }_{L^2(\R)}^2   \Id s
    - \int\limits_0^t \bigg \langle \rho_s \frac{\dd}{\dd x} (\sigma_s^2) , \frac{\dd}{\dd x} \rho_s \bigg \rangle_{L^2(\R)}  \Id s \\
    &\quad\quad -  2 \int\limits_0^t  \bigg \langle (k *\rho_s)\rho_s ,\frac{\dd}{\dd x} \rho_s \bigg \rangle _{L^2(\R)}   \Id s  \\
    &\quad\le  -\frac{\lambda}{2} \int\limits_0^t  \norm{\frac{\dd}{\dd x} \rho_s }_{L^2(\R)}^2  \Id s
    +  \sup\limits_{0 \le t \le T} \norm{\frac{\dd}{\dd x} (\sigma^2_t) }_{L^\infty(\R)} \int\limits_0^t \norm{\rho_s }_{L^2([-L,L])} \norm{  \frac{\dd}{\dd x} \rho_s}_{L^2(\R)}  \Id s\\
    &\quad\quad + \frac{1}{\lambda} \int\limits_0^t \norm{(k *\rho_s)\rho_s}_{L^2(\R)}^2 \Id s   \\
    &\quad\le \Big(-\frac{\lambda}{2} + 2L^2  \sup\limits_{0 \le t \le T}  \norm{\frac{\dd}{\dd x} (\sigma^2_t)}_{L^\infty(\R)} \Big) \int\limits_0^t \norm{ \frac{\dd}{\dd x} \rho_s}_{L^2(\R)}^2 \Id s\\
    &\quad\quad+ \frac{1}{\lambda}   \int\limits_0^t \norm{k *\rho_s}_{L^4(\R)}^2 \norm{\rho_s}_{L^4(\R)}^2 \Id s   \\
    &\quad\le  \bigg( -\frac{\lambda}{2} + 2L^2  \sup\limits_{0 \le t \le T} \norm{\frac{\dd}{\dd x}  (\sigma^2_t)}_{L^\infty(\R)} \bigg) \int\limits_0^t \norm{\frac{\dd}{\dd x}  \rho_s }_{L^2(\R)}^2 \Id s \\
    &\quad\quad+ \frac{1}{\lambda}   \int\limits_0^t \norm{k}_{L^1(\R)}^2  \norm{\rho_s}_{L^4(\R)}^2 \norm{\rho_s}_{L^4(\R)}^2 \Id s  \\
    &\quad= \bigg( -\frac{\lambda}{2} + 2L^2  \sup\limits_{0 \le t \le T} \norm{\frac{\dd}{\dd x} (\sigma^2_t)}_{L^\infty(\R)} \bigg) \int\limits_0^t \norm{\frac{\dd}{\dd x} \rho_s }_{L^2(\R)}^2 \Id s
    + \frac{  \norm{k}_{L^1(\R)}^2 }{\lambda} \int\limits_0^t    \norm{\rho_s}_{L^4(\R)}^4 \Id s .
  \end{align*}
  We remark that we used integration by parts in the first step, Young's and H{\"o}lder's inequality in the third step, H{\"o}lder's and Poincar{\'e} inequaity \cite[Theorem~13.19]{LeoniGiovanni2017Afci} in the forth step and Young's inequality for convolutions in the fifth step. Let us recall the Gagliardo--Nirenberg--Sobolev interpolation \cite[Theorem~12.83]{LeoniGiovanni2017Afci}, which states that for \(u \in L^1(\R) \cap W^{1,2}(\R)\) we have
  \begin{equation*}
    \norm{u}_{L^4(\R)} \le C_{GNS} \norm{u}_{L^1(\R)}^{1/2} \norm{\frac{\dd}{\dd x} u}_{L^{2}(\R)}^{1/2} .
  \end{equation*}
  Consequently, applying this inequality on the last term in our estimate and having mass conservation in mind we find
  \begin{align*}
    &\norm{\rho_t}_{L^2(\R)}^2 - \norm{\rho_0}_{L^2(\R)}^2\\
    &\quad\le \bigg( -\frac{\lambda}{2} + 2L^2  \sup\limits_{0 \le t \le T} \norm{\frac{\dd}{\dd x} (\sigma^2_t)}_{L^\infty(\R)} + \frac{C_{GNS}^4 \norm{k}_{L^1(\R)}^2 }{\lambda}  \bigg) \int\limits_0^t \norm{\frac{\dd}{\dd x} \rho_s }_{L^2(\R)}^2 \Id s   .
  \end{align*}
  Hence, if
  \begin{equation*}
    2 L^2 \sup\limits_{0 \le t \le T} \norm{\frac{\dd}{\dd x} (\sigma^2_t)}_{L^\infty(\R)}
    + \bigg(C_{GNS}^4 \norm{k}_{L^1(\R)}^2 + 4 L^4 \sup\limits_{0 \le t \le T} \norm{\frac{\dd}{\dd x} (\sigma^2_t)}_{L^\infty(\R)}^2\bigg)^{1/2} \le \lambda,
  \end{equation*}
  we discover
  \begin{equation}\label{eq: decresing_L_2_norm}
    \norm{\rho}_{S_{\cF^W}^\infty([0,T^*];L^2(\R))} \le \norm{\rho_0}_{L^2(\R)}^2  .
  \end{equation}
  Since \(\rho \in L_{\cF^W}^2([0,T^{*}];W^{1,2}(\R) ) \), we may apply \cite[Theorem~7.1]{Krylov1999AnAA}, which tells us that \(\rho \in C([0,T^{*}], L^2(\R))\), $\P$-a.s., and
  \begin{equation*}
    \E(\norm{\rho_{T^{*}}}_{L^2(\R)}^2) < \infty.
  \end{equation*}
  As a result, we can take \(\rho_{T^{*}}\) as the new initial value and apply \cite[Theorem~5.1]{Krylov1999AnAA} in combination with our above arguments in proof of Theorem~\ref{theorem: existence_of_hk_spde} to obtain a solution on \([T^{*}, 2T^{*}]\), since the estimate \eqref{eq: decresing_L_2_norm} and the condition \eqref{eq: L_2_condition} are independent of \(T^*\). Hence, after finitely many steps we have a global solution~\(\rho\) on \([0,T]\). The uniqueness and \(\rho \in \mathbb{B}\) follows by repeating the inequalities derived in the contraction argument in Theorem~\ref{theorem: existence_of_hk_spde} or using the uniqueness of the SPDE presented in \cite[Theorem~5.1 and Corollary~5.11]{Krylov1999AnAA}.
\end{proof}

\begin{remark}
  In particular, for a constant diffusion \(\sigma > 0\) the condition~\eqref{eq: L_2_condition} reads simply as
  \begin{equation*}
    C_{GNS}^2 \norm{k}_{L^1(\R)} \le \sigma,
  \end{equation*}
  which can be interpreted such that for a given integrable kernel \(k\) the system needs a certain amount of idiosyncratic noise to stay alive for arbitrary \(T>0\). In other word, the diffusion needs to be dominant. 
\end{remark}

Next, we are going to improve the regularity of the solution \(\rho\) by a bootstrap argument.

\begin{lemma} \label{lem: regularity_general_SPDE}
  Let \(\rho_0 \in L^1(\R) \cap W^{2,2}(\R)\) with \(\norm{\rho_0}_{L^1(\R)} = 1 \). Moreover, let Assumption~\ref{ass: sigma_ellipticity+bound} hold and \(k \in  L^2(\R)\). Assume we have a solution
  \begin{equation*}
    \rho \in L^2_{\cF^W}([0,T];W^{1,2}(\R)) \cap S^\infty_{\cF^W}([0,T];L^1(\R) \cap L^{2}(\R))
  \end{equation*}
  of the SPDE~\eqref{eq: hk_spde general} on \([0,T]\). Then \(\rho\) has the following regularity
  \begin{equation*}
    \rho \in L^2_{\cF^W}([0,T];W^{3,2}(\R)) \cap S^2_{\cF^W}([0,T],W^{2,2}(\R)) \cap S^\infty_{\cF^W}([0,T];L^1(\R) \cap L^{2}(\R)).
  \end{equation*}
\end{lemma}

\begin{proof}
  Let us explore the following bootstrap argument. By assumptions we know \(\rho \in L^2_{\cF^W}([0,T];W^{1,2}(\R) ) \cap S^\infty_{\cF^W}([0,T];L^1(\R) \cap L^2(\R))\) and solves the SPDE
  \begin{equation}\label{eq: proof_SDE_existence_SPDE_hk}
    \Id \rho_t = \frac{\dd^2}{\dd x^2} \Big( \frac{\sigma^2_t + \nu^2 }{2}  \rho_t \Big) \Id t + \frac{\dd}{\dd x} ((k *\rho_t)\rho_t) \Id t - \nu \frac{\dd}{\dd x} \rho_t \Id W_t .
  \end{equation}
  Furthermore, \(\frac{\dd}{\dd x} (k^\tau *\rho_t) = k^\tau * \frac{\dd}{\dd x} \rho_t\) for the smooth approximation \(k^\tau\) of \(k\) and consequently the dominated convergence theorem implies \(\frac{\dd}{\dd x} (k*\rho_t) = k * \frac{\dd}{\dd x} \rho_t\) in the sense of distributions. As a result
  \begin{equation*}
    \frac{\dd}{\dd x} ((k *\rho_t)\rho_t)  = \bigg(k * \frac{\dd}{\dd x} \rho_t\bigg)\rho_t + (k *\rho_t) \frac{\dd}{\dd x} \rho_t
  \end{equation*}
  is well-defined as a function in \(L^1(\R)\). Moreover, we find
  \begin{align*}
    \begin{split}
    \norm{\frac{\dd}{\dd x} ((k *\rho_t)\rho_t)}_{L^2(\R)}
    &\le \norm{\bigg(k * \frac{\dd}{\dd x} \rho_t\bigg)\rho_t}_{L^2(\R)} + \norm{(k *\rho_t) \frac{\dd}{\dd x} \rho_t}_{L^2(\R)} \\
    &\le \norm{ k * \frac{\dd}{\dd x} \rho_t }_{L^\infty(\R)} \norm{\rho_t}_{L^2(\R)}  + \norm{ k *  \rho_t }_{L^\infty(\R)} \norm{\frac{\dd}{\dd x}\rho_t}_{L^2(\R)} \\
    &\le \norm{ k}_{L^2(\R)} \norm{\rho_t}_{W^{1,2}(\R)} \norm{\rho_t}_{L^2(\R)}  + \norm{k}_{L^2(\R)} \norm{\rho_t}_{L^2(\R)}  \norm{\rho_t}_{W^{1,2}(\R)}  \\
    &\le 2\norm{ k}_{L^2(\R)}  \norm{\rho_t}_{W^{1,2}(\R)} \norm{\rho_t}_{L^2(\R)},
    \end{split}
  \end{align*}
  which implies
  \begin{equation*}
    \norm{\frac{\dd}{\dd x} ((k *\rho)\rho)}_{L^2_{\cF^W}([0,T];L^2(\R))} \le 2\norm{ k}_{L^2(\R)} \norm{\rho}_{S^\infty_{\cF^W}([0,T];L^2(\R))} \norm{\rho}_{L^2_{\cF^W}([0,T];W^{1,2}(\R) )}.
  \end{equation*}
  From the uniqueness of the SPDE~\eqref{eq: proof_SDE_existence_SPDE_hk}, \(\rho_0 \in W^{1,2}(\R)\) and \cite[Theorem~5.1 and Theorem~7.1]{Krylov1999AnAA} we obtain
  \begin{equation*}
    \rho \in L^2_{\cF^W}([0,T];W^{2,2}(\R)) \cap S^2_{\cF^W}([0,T];W^{1,2}(\R)).
  \end{equation*}
  With the same strategy and the discovered regularity of \(\rho\) one obtains
  \begin{align*}
    \frac{\dd^2}{\dd x^2} ((k *\rho_t)\rho_t)  =\bigg (k * \frac{\dd^2}{\dd x^2} \rho_t\bigg)\rho_t +  2 \bigg(k * \frac{\dd}{\dd x} \rho_t\bigg)  \frac{\dd}{\dd x}\rho_t  + (k *\rho_t) \frac{\dd^2}{\dd x^2} \rho_t
  \end{align*}
  and consequently
  \begin{align*}
    \norm{\frac{\dd^2}{\dd x^2} ((k *\rho_t)\rho_t)}_{L^2(\R)}
    &\le 2\norm{ k}_{L^2(\R)} \norm{\rho_t}_{W^{2,2}(\R)} \norm{\rho_t}_{L^2(\R)} + 2\norm{ k}_{L^2(\R)}  \norm{ \frac{\dd}{\dd x} \rho_t}_{L^{2}(\R)}^2 \\
    &\le 4 \norm{ k}_{L^2(\R)}  \norm{\rho_t}_{W^{2,2}(\R)} \norm{\rho_t}_{L^2(\R)},
  \end{align*}
  where we used  Gagliardo--Nirenberg--Sobolev interpolation inequality \cite[Theorem~7.41]{LeoniGiovanni2017Afci} in the last step. Therefore, we have
  \begin{equation*}
    \norm{\frac{\dd}{\dd x} ((k *\rho)\rho)}_{L^2_{\cF^W}([0,T];W^{1,2}(\R))}
    \le 4 \norm{ k}_{L^2(\R)}  \norm{\rho}_{S^\infty_{\cF^W}([0,T];L^2(\R))} \norm{\rho}_{L^2_{\cF^W}([0,T];W^{2,2}(\R) )}.
  \end{equation*}
  Again, from the uniqueness of the SPDE~\eqref{eq: proof_SDE_existence_SPDE_hk}, \(\rho_0 \in W^{2,2}(\R)\) and \cite[Theorem~5.1, Corollary~5.11, Theorem~7.1]{Krylov1999AnAA} we obtain
  \begin{equation*}
    \rho \in L^2_{\cF^W}([0,T];W^{3,2}(\R)) \cap S^2_{\cF^W}([0,T],W^{2,2}(\R)).
  \end{equation*}
\end{proof}

As a consequence of Theorem~\ref{theorem: existence_of_hk_spde}, Corollary~\ref{cor: spde_global_weak_solution} and the fact that \(k_{\scriptscriptstyle{HK}},k_{\scriptscriptstyle{HK}}^\tau \in L^1(\R) \cap L^2(\R) \) for all \(\tau > 0\), we obtain the following corollary.  

\begin{corollary} \label{cor: hk_spde_well_pos}
  Let Assumption~\ref{ass: sigma_ellipticity+bound} hold. Further, assume \(0 \le \rho_0 \in L^1(\R) \cap L^2(\R) \) with \(\norm{\rho_0}_{L^1(\R)}=1\). Then, there exists a \(T^*>0\) and a unique non-negative solution \(\rho\), \(\rho^\tau\) of the SPDE~\eqref{eq: hk_spde} and~\eqref{eq: regularised_hk_spde} in the space
  \begin{equation*}
    \mathbb{B}  =  L^2_{\cF^W}([0,T^*];W^{1,2}(\R) ) \cap S^\infty_{\cF^W}([0,T^*];L^1(\R) \cap L^2(\R))
  \end{equation*}
  Furthermore, if \(\frac{\dd}{\dd x} \sigma_t(x) \) has compact support and inequality~\eqref{eq: L_2_condition} holds, then we can extend \(\rho\), \(\rho^\tau\) to a global solution.
\end{corollary}

\section{Well-posedness of the mean-field SDEs}\label{sec: mean_field_sde}

In this section we establish the existence of unique strong solutions of the mean-field stochastic differential equations~\eqref{eq: mean_field_trajectories_with_common_noise} and~\eqref{eq: regularized_mean_field_trajectories_with_common_noise}. Analogously to the classical theory of ordinary SDEs, it turns out that the mean-field SDEs~\eqref{eq: mean_field_trajectories_with_common_noise} and~\eqref{eq: regularized_mean_field_trajectories_with_common_noise} are linked to the stochastic Fokker--Planck equations~\eqref{eq: hk_spde} and \eqref{eq: regularised_hk_spde}. The same connection between ordinary stochastic differential equations (SDEs) and deterministic Fokker--Planck equations (also known as Kolmogorov forward equations) has been extensively studied, for instance by~\cite{Stewart1986,Figali2008,Trevisan2015,barburoeckner2020}.

\medskip 

Similar to Section~\ref{sec:FP equation}, we prove the existence of strong solutions for general interaction force~\(k\). In the following we consider the mean-field SDE
\begin{align}\label{eq: mean_field_trajectories_with_common_noise general}
  \begin{cases}
  \Id Y_t = -  (k * \rho_t)(Y_t) \Id t + \sigma(t,Y_t) \Id B_t + \nu \Id W_t, \quad Y_0=X_0 , \\
  \rho_t \; \mathrm{is \, the \, conditional \, density \, of} \, Y_t \, \mathrm{given} \, \cF_t^W,
  \end{cases}
\end{align}
for a general interaction force \(k \in L^1(\R) \cap L^2(\R)\). We notice that~\eqref{eq: mean_field_trajectories_with_common_noise general} is just one of the identically distributed SDEs of the system \eqref{eq: mean_field_trajectories_with_common_noise}, if we set \(k =k_{\scriptscriptstyle{HK}}\). In order to guarantee the well-posedness of the SPDE~\eqref{eq: hk_spde general} we make the assumption:

\begin{assumption}\label{ass: spde_existence}
  Let \(0 \le \rho_0 \in L^1(\R) \cap W^{2,2}(\R) \) with \(\norm{\rho_0}_{L^1(\R)}=1\). For \(T>0\) there exists a unique solution~$\rho$ in \(L^2_{\cF^W}([0,T];W^{1,2}(\R))\) of the SPDE~\eqref{eq: hk_spde general} on the interval \([0,T]\) with
  \begin{align*}
    \norm{\rho}_{L^2_{\cF^w}([0,T];W^{1,2}(\R))} + \norm{\rho}_{S^\infty_{\cF^w}([0,T];L^1(\R) \cap L^2(\R))} &\le C
  \end{align*}
  for some finite constant \(C>0\).
\end{assumption}

\begin{remark}
  The existence of a unique solution to the SPDE~\eqref{eq: hk_spde general} in the above assumption is satisfied, for instance, if the conditions stated in Remark~\ref{remark: spde_global_solution_weak_existence}, Theorem~\ref{theorem: existence_of_hk_spde} or Corollary~\ref{cor: spde_global_weak_solution} are satisfied.
\end{remark}

\begin{theorem}\label{theorem: existence_mean_field_SDE}
  Let Assumption~\ref{ass: sigma_ellipticity+bound} as well as Assumption~\ref{ass: spde_existence} hold and \(k \in L^1(\R) \cap L^2(\R)\). Then, the mean-field SDE~\eqref{eq: mean_field_trajectories_with_common_noise general} has a unique strong solution \((Y_t, t \in [0,T])\) and \(\rho_t\) is the conditional density of \(Y_t\) given \(\cF_t^{W}\) for every $t\in [0,T]$. 
\end{theorem}

The idea to prove Theorem~\ref{theorem: existence_mean_field_SDE} is to freeze the measure \(\rho_t\) in the SDE~\eqref{eq:  mean_field_trajectories_with_common_noise general} and use a duality argument by introducing a dual backward stochastic partial differential equation (BSPDE) in Lemma~\ref{lemma: dual_BSPDE} in order to prove that \(\rho_t\) is the conditional density of \(Y_t\) for given \(\cF^W_t\).

\begin{proof}
  Let $\rho$ be the unique solution of the SPDE~\eqref{eq: hk_spde general} as in Assumption~\ref{ass: spde_existence}. We recall that by the regularity result presented in Lemma~\ref{lem: regularity_general_SPDE} we have
  \begin{equation*}
    \rho \in L^2_{\cF^W}([0,T];W^{3,2}(\R)) \cap S^2_{\cF^W}([0,T],W^{2,2}(\R)).
  \end{equation*}

  \textit{Step~1.} Fix $\rho$ in the mean-field SDE~\eqref{eq: mean_field_trajectories_with_common_noise general} and notice that we are dealing with a standard SDE with random coefficients. Hence, we can apply classical results if the drift coefficient \(k*\rho\) is Lipschitz continuous. The regularity of the solution, Sobolev embedding theorem~\cite[Theorem 8.8]{BrezisHaim2011FaSs} and Morrey's inequality yields~\cite[Theorem 12.66]{LeoniGiovanni2017Afci}
  \begin{align*}
    \sup\limits_{0 \le t \le T } \sup\limits_{\substack{x,y \in \R \\ x \neq y }} \frac{|(k*\rho_t)(\omega,x)-(k*\rho_t)(\omega,y) |}{|x-y|}
    &\le \sup\limits_{0 \le t \le T } \norm{k*\rho_t(\omega)}_{W^{2,2}(\R)} \\
    &\le \norm{ k}_{L^1(\R)} \sup\limits_{0 \le t \le T }   \norm{\rho_t(\omega)}_{W^{2,2}(\R)}
  \end{align*}
  and
  \begin{align*}
    \sup\limits_{0 \le t \le T } |(k*\rho_t)(\omega,0) |
    &\le \sup\limits_{0 \le t \le T }  \sup\limits_{x \in \R } |(k*\rho_t)(\omega,x) | \\
    &\le \norm{ k}_{L^1(\R)} \sup\limits_{0 \le t \le T }   \norm{\rho_t(\omega)}_{W^{1,2}(\R)} .
  \end{align*}
  Furthermore, the maps \(\omega \mapsto  \sup\limits_{0 \le t \le T }   \norm{\rho_t(\omega)}_{W^{2,2}(\R) }\) and \(\omega \mapsto  \sup\limits_{0 \le t \le T }   \norm{\rho_t(\omega)}_{W^{1,2}(\R) }\) are measurable. Therefore, standard results for the existence of SDEs with Lipschitz continuous drift, see e.g. \cite[Theorem~1.1]{KrylovRocknerZabcyk1998} or \cite[Theorem~2.2]{random_coeff_SDE}, imply that the following SDE has a unique strong solution
  \begin{equation}\label{eq: linearised_mean_field_sde_with_enviromental_noise}
  \begin{cases}
    \Id \overline{Y}_t =  -  (k * \rho_t)(\overline{Y}_t) \Id t + \sigma(t,\overline{Y}_t ) \Id B_t + \nu \Id W_t, \\
    \overline{Y}_0  \sim \rho_0 \; . 
    \end{cases}
  \end{equation}

  \textit{Step~2.} We are going to use a dual argument (see Lemma~\ref{lemma: dual_BSPDE} below) to show that \(\rho_t\) is the conditional density of $Y_t$ with respect to \(\cF_t^W\). Hence, let \(T_1 > 0\) and \((u_t, t \in [0,T_1])\) be the solution of the BSPDE~\eqref{eq: bspde_hk_model} below with terminal condition \(G \in L^\infty(\Omega, \cF_{T_1} , \testfunctions{\R}) \). Utilizing the dual equation from Lemma~\ref{lemma: dual_BSPDE}, the dual analysis \cite[Corollary~3.1]{zhou1992duality} and the fact that \(u_0\) is \(\cF_0^W\)-measurable, we find
  \begin{equation*}
    \qv{\rho_0, u_0} = \E(\qv{G,\rho_{T_1}}) .
  \end{equation*}
  On the other hand we can use the explicit representation of \(u_0\) given by Lemma~\ref{lemma: dual_BSPDE} to obtain
  \begin{align*}
    \qv{\rho_0, u_0}
    = \int_\R u_0(y) \rho_0(y) \Id y
    = \E(u_0(\overline{Y}_0))
    = \E( \E( G(\overline{Y}_{T_1}) | \,\sigma( \cF_0^W ,\sigma(\overline{Y}_0) ) ) ) = \E(G(\overline{Y}_{T_1})).
  \end{align*}
  Now, let \(G=\phi \xi\) with \(\phi \in \testfunctions{\R}\) and \(\xi \in L^\infty(\Omega, \cF_{T_1})\). Consequently, we obtain
  \begin{equation*}
    \E(\xi \qv{\phi,\rho_{T_1}}) = \E( \xi \phi(\overline{Y}_{T_1})) = \E(\xi  \E( \phi(\overline{Y}_{T_1}) \, | \, \cF_{T_1}^W)),
  \end{equation*}
  which proves
  \begin{equation*}
    \qv{\phi,\rho_{T_1}} =  \E( \phi(\overline{Y}_{T_1}) \, | \, \cF_{T_1}^W), \quad \P\text{-a.e.},
  \end{equation*}
  and, therefore, \(\rho_{T_1}\) is the conditional density of \(\overline{Y}_{T_1}\) given \(\cF_{T_1}\). Since \(T_1\) is arbitrary, we have proven the existence of a strong solution \(Y\) of the mean-field SDE~\eqref{eq: mean_field_trajectories_with_common_noise general}.

  On the other hand, if \eqref{eq: mean_field_trajectories_with_common_noise general} has a strong solution with conditional density
  \begin{equation*}
    \rho \in L^2_{\cF^W}([0,T^*];W^{1,2}(\R) ) \cap S^\infty_{\cF^W}([0,T^*];L^1(\R) \cap L^2(\R)),
  \end{equation*}
  then the conditional density process of \(Y^1\) is the solution to the SPDE~\eqref{eq: hk_spde general}. Indeed, if we first apply It{\^o}'s formula with a function \(\varphi \in \testfunctions{\R}\), then take the conditional expectation with respect to the filtration \(\cF^W\) and subsequently applying stochastic Fubini theorem \cite[Lemma~A.5]{Hammersley2020}, we conclude that density process of \(Y_t\) satisfies \eqref{eq: solution_hk_spde}. By the uniqueness of the SPDE~\eqref{eq: hk_spde general} we obtain that \(\rho_t\), which is the solution constructed in Theorem~\ref{theorem: existence_of_hk_spde}, is the conditional density of \(Y_t\) given \(\cF^W_t\) for all \(t \in [0,T]\).
\end{proof}

In the following lemma, we close the gap in the above proof by demonstrating the existence of a solution of the BSPDE~\eqref{eq: bspde_hk_model} and the explicit representation of \(u_0\).

\begin{lemma}[Dual BSPDE]\label{lemma: dual_BSPDE}
  Let Assumption~\ref{ass: sigma_ellipticity+bound} and Assumption~\ref{ass: spde_existence} hold along with \(k \in L^1(\R) \cap L^2(\R)\). Then, for every  \(T_1 \in (0,T]\) and \(G \in L^\infty(\Omega, \cF_{T_1} , \testfunctions{\R}) \) the following BSPDE
  \begin{align}\label{eq: bspde_hk_model}
     \begin{split}
    \Id u_t &= - \bigg( \frac{\sigma^2_t+\nu^2}{2} \frac{\dd^2}{\dd x^2} u_t -(k*\rho_t) \frac{\dd}{\dd x} u_t  + \nu \frac{\dd}{\dd x} v_t \bigg) \Id t   + v_t \Id W_t,\quad t\in [0,T],  \\
    u_{T_1} &= G,
    \end{split}
  \end{align}
  admits a unique solution
  \begin{equation*}
    (u,v) \in (  L^2_{\cF^W}([0,T]; W^{2,2}(\R))  \cap S^2_{\cF^W}([0,T];W^{1,2}(\R))) \times L^2_{\cF^W}([0,T];W^{1,2}(\R)),
  \end{equation*}
  i.e. for any \(\varphi \in \testfunctions{\R}\) the equality
  \begin{align*}
    \qv{u_t,\varphi}_{L^2(\R)} = & \;   \qv{G,\varphi}_{L^2(\R)}
    + \int\limits_t^{T_1} \bigg\langle   \frac{\sigma^2_s+\nu^2}{2} \frac{\dd^2}{\dd x^2} u_s -(k*\rho_s) \frac{\dd}{\dd x} u_s  + \nu \frac{\dd}{\dd x} v_s, \varphi \bigg\rangle_{L^2(\R)} \Id s  \\
    & - \int\limits_t^{T_1}  \qv{  v_s ,    \varphi}_{L^2(\R)} \Id W_s
  \end{align*}
  holds for all \(t \in [0,T]\) with probability one. Moreover, we have
  \begin{equation}\label{eq: bspde_charc_conditional_expectation}
    u_0(\overline{Y}_0) = \E(G(\overline{Y}_{T_1}) \, | \, \sigma(\sigma(\overline{Y}_0), \cF_0^W )),
  \end{equation}
  where \((\overline{Y}_t, t \in [0,T])\) is the solution of the linearised SDE~\eqref{eq: linearised_mean_field_sde_with_enviromental_noise} in the proof of Theorem~\ref{theorem: existence_mean_field_SDE}.
\end{lemma}

\begin{proof}
  We note that by Theorem~\ref{theorem: existence_mean_field_SDE} we have \(\rho \in  L^2_{\cF}([0,T];W^{3,2}(\R)) \cap S^2_{\cF}([0,T];W^{2,2}( \R)) \). Our approach is to verify the assumptions of the \(L^2\)-theory (see for example \cite[Theorem~5.5]{kai_du_L_p_bspde_2011}) for BSPDEs. Let \(u_1, u_2 \in W^{2,2}(\R)\), then
  \begin{align*}
    \norm{(k*\rho_t) \frac{\dd}{\dd x}u_1- (k*\rho_t) \frac{\dd}{\dd x} u_2}_{L^{2}(\R)}
    &\le \norm{k*\rho_t}_{L^\infty(\R)} \norm{ \frac{\dd}{\dd x} (u_1-u_2)}_{L^2(\R)}  \\
    &\le \norm{k}_{L^2(\R)} \norm{\rho_t}_{L^2(\R)}  \norm{ \frac{\dd}{\dd x} (u_1-u_2)}_{L^2(\R)} \\
    &\le  \norm{k}_{L^2(\R)}   \norm{\rho_t}_{L^2(\R)}  \norm{u_1-u_2}_{W^{1,2}(\R)} .
  \end{align*}
  Now, by Theorem~\ref{theorem: existence_of_hk_spde}, \( \norm{\rho_t}_{L^2(\R)}  \) is uniformly bounded in \((\omega,t) \in \Omega \times [0,T]\) and the interpolation theorem \cite[Theorem~5.2]{AdamsRobertA2003Ss} implies for all \(\epsilon>0\),
  \begin{align*}
    &\norm{(k*\rho_t) \frac{\dd}{\dd x}u_1- (k*\rho_t) \frac{\dd}{\dd x} u_2}_{L^{2}(\R)}\\
    &\quad\le \epsilon  \norm{u_1-u_2}_{W^{2,2}(\R)} + C(k, \norm{\rho}_{S^\infty_{\cF^W}([0,T]; L^2(\R))} )  \kappa(\epsilon)  \norm{u_1-u_2}_{L^{2}(\R)}
  \end{align*}
  for some non-negative decreasing function~\(\kappa\). Hence, Assumption~5.4 in \cite[Theorem~5.5]{kai_du_L_p_bspde_2011} is satisfied. The other assumptions are also easily verified. As a result we obtain a solution
  \begin{equation*}
    (u,v) \in (  L^2_{\cF^W}([0,T]; W^{2,2}(\R))  \cap S^2_{\cF^W}([0,T];L^{2}(\R))) \times L^2_{\cF^W}([0,T];W^{1,2}(\R))
  \end{equation*}
  of the BSPDE~\eqref{eq: bspde_hk_model}. Here, the fact that \(u \in  S^2_{\cF^W}([0,T];L^{2}(\R))\) is a direct consequence of \cite[Theorem~2.2]{DU2BSPDE2010}.

  It remains to show that the equality~\eqref{eq: bspde_charc_conditional_expectation} holds. By the bound
  \begin{equation*}
    \E\bigg( \sup\limits_{t \le T_1} \norm{\rho_{t}}_{W^{1,2}(\R)}^2 \bigg)
    < \infty
  \end{equation*}
  given by \(\rho \in S^2_{\cF^W}([0,T];W^{1,2}(\R))\), we observe that there exists a set \(\Omega'\) with \(\P(\Omega')=1\) and for all \(\omega \in \Omega'\) we have
  \begin{equation}\label{eq: u_a_e_derivative_bound}
    \sup\limits_{t \le T_1} \norm{\rho_t(\omega,\cdot)}_{W^{1,2}(\R)} < \infty.
  \end{equation}
  Also, the map \((\omega, t) \to \norm{\rho_t(\omega,\cdot)}_{W^{1,2}(\R)}\) is predictable with respect to \(\cF^W\) by the \(L^2 \)-SPDE theory. Consequently, we can define for each \(m \in \N\) the stopping time
  \begin{equation*}
    \tau_m(\omega) = \inf \{ t \in [0,T_1]  :   \norm{\rho_t(\omega,\cdot)}_{W^{1,2}(\R)} \ge m \}
  \end{equation*}
  and \(\tau_m \uparrow T_1\) by \eqref{eq: u_a_e_derivative_bound}. Furthermore, let us define
  \begin{equation*}
    F(t,x) := (k*\rho_t)(x) \frac{\dd}{\dd x} u_t(x)
    \quad \text{and}\quad
    F_m(t,x) := F(t,x) \indicator{(0,\tau_m]}(t),
  \end{equation*}
  and note that \(F_m \in L^2_{\cF^W}([0,T]; L^{2}(\R)) \) still satisfies all assumptions of the \(L^2\)-BSPDE theory (\cite[Theorem~5.5]{kai_du_L_p_bspde_2011}) and therefore there exists a solution
  \begin{equation*}
    (u^m, v^m) \in (  L^2_{\cF^W}([0,T]; W^{2,2}(\R))  \cap S^2_{\cF^W}([0,T];L^{2}(\R))) \times L^2_{\cF^W}([0,T];W^{1,2}(\R))
  \end{equation*}
  of the following BSPDE
  \begin{align*}
    \begin{split}
    \Id u^m_t &= - \bigg( \frac{\sigma^2_t+\nu^2}{2} \frac{\dd^2}{\dd x^2} u^m_t -F_m(t)  + \nu \frac{\dd}{\dd x} v^m_t \bigg ) \Id t   + v^m_t \Id W_t , \\
    u^m_{T_1} &= G,
    \end{split}
  \end{align*}
  for each \(m \in \N\).

  In the next step we want to obtain a \(L^2_{\cF^W}([0,T]; W^{1,2}(\R))\)-bound for \(F_m\). The \(L^2\)-estimate follows directly from the above computations. For the weak derivative we compute
  \begin{align*}
    &\norm{\frac{\dd}{\dd x}\bigg( (k*\rho_t) \frac{\dd}{\dd x}u_t \bigg) \indicator{(0,\tau_m]} }_{L^2(\R)}\\
    &\quad\le \norm{ \indicator{(0,\tau_m]} \bigg(k* \frac{\dd}{\dd x}\rho_t\bigg) \frac{\dd}{\dd x} u_t}_{L^2(\R) }  + \norm{ (k*\rho_t) \frac{\dd^2}{\dd x^2} u_t}_{L^2(\R)} \\
    &\quad\le  \indicator{(0,\tau_m]} \norm{ \bigg(k* \frac{\dd}{\dd x}\rho_t\bigg) \frac{\dd}{\dd x}u_t}_{L^2(\R)}  + \norm{k}_{L^2(\R)} \norm{\rho_t}_{S^\infty_{\cF}([0,T];L^{2}(\R)))}  \norm{u_t}_{W^{2,2}(\R)}.
  \end{align*}
  Since \(\rho \in S^\infty_{\cF^W}([0,T];L^{2}(\R)))\), the last term behaves nicely. However, the first term would be problematic because without the stopping time we do not have a similar \(L^\infty\)-estimate for the derivative, i.e. \(\rho \in S^\infty_{\cF^W}([0,T];W^{1,2}(\R)))\). Hence, in order to overcome this problem we introduced the stopping time \(\tau_m\) and, therefore, we discover
  \begin{align*}
    &\norm{\frac{\dd}{\dd x}\bigg( (k*\rho_t) \frac{\dd}{\dd x}u_t\bigg) \indicator{(0,\tau_m]} }_{L^2_{\cF^W} ([0,T_1];L^2(\R))} \\
    &\quad\le   \norm{k}_{L^2(\R)}^2 \E \bigg( \int\limits_{0}^{T_1} \indicator{(0,\tau_m]}(t) \norm{\frac{\dd}{\dd x} \rho_t}_{L^2(\R)}^2 \norm{\frac{\dd}{\dd x}u_t}_{L^2(\R)}^2  \,\mathrm{d} t  \bigg)\\
    &\quad\quad+ \norm{k}_{L^2(\R)} \norm{\rho_t}_{S^\infty_{\cF}([0,T];L^{2}(\R)))}  \norm{u_t}_{L^2_{\cF^W}([0,T]; W^{2,2}(\R)) }   \\
    &\quad\le  \norm{k}_{L^2(\R)} m^2  \norm{u}_{L^2_{\cF^W}([0,T]; W^{1,2}(\R)) } \\
    &\quad\quad+ \norm{k}_{L^2(\R)} \norm{\rho_t}_{S^\infty_{\cF}([0,T];L^{2}(\R)))}  \norm{u_t}_{L^2_{\cF^W}([0,T]; W^{2,2}(\R)) } .
  \end{align*}
  As a result, we obtain
  \begin{equation*}
    \norm{F_m}_{L^2_{\cF^W}([0,T];W^{1,2}(\R))} < \infty
  \end{equation*}
  for each \(m \in \N\). Applying \cite[Theorem~5.5]{kai_du_L_p_bspde_2011} again, we find
  \begin{equation*}
    (u^m,v^m ) \in (  L^2_{\cF^W}([0,T]; W^{3,2}(\R))  \cap S^2_{\cF^W}([0,T];W^{1,2}(\R))) \times L^2_{\cF^W}([0,T];W^{2,2}(\R)).
  \end{equation*}
  The above regularity (\(p(m-2) > 1\) with \(p=2,m=3\)) allows us to apply \cite[Corollary~2.2]{DuTang2013}, which tells us that there exists a set of full measure \(\Omega_m''\) maybe different from \(\Omega'\) such that
  \begin{align*}
    u^m(t,x) = &\,  G(x) +  \int\limits_{t}^{T_1} \frac{\sigma^2_t+\nu^2 }{2} \frac{\dd^2}{\dd x^2} u^m(s,x) - \indicator{(0,\tau_m]}(s) (k*\rho)(s,x) \frac{\dd}{\dd x} u(s,x)  \\
    &+ \nu \frac{\dd}{\dd x} v^m(s,x) \Id s  -  \int\limits_{t}^{T_1} v^m(s,x) \Id W_s
  \end{align*}
  holds for all \((t,x) \in [0,T_1] \times \R \) on \(\Omega_m''\). We use the subscript \(m\) to indicate that even though the set \(\Omega_m''\) is independent of \((t,x)\) it still may depend on \(m \in \N\).

  Besides, to be precise, \cite[Corollary~2.2]{du2018w2psolutions} actually requires \(F_m \in L^2_{\cF^W}([0,T];W^{3,2}(\R))\), which is more regularity than we have. However, one can modify the proof of \cite[Corollary~2.2]{du2018w2psolutions} to obtain the same result with \(F_m \in L^2_{\cF^W}([0,T];W^{1,2}(\R))\). The crucial part is that a mollification of \(F_m\) with the standard mollifier converges in the supremum norm to~\(F_m\), which follows from Morrey's inequality even in our case \( F_m \in L^2_{\cF^W}([0,T];W^{1,2}(\R))\). For a similar result in the SPDE setting we refer to \cite[Lemma~4.1]{RozovskyBorisL2018SES}.
 
  Next, we want to apply an It{\^o}--Wentzell type formula \cite[Theorem~3.1]{Yang2013DynkinGO} (with \(V=u, \; X=\overline{Y}\) therein). Hence, we need to verify the required assumption. First, we can view \(u_m\) as a jointly continuous It{\^o} process in \((t,x)\) by \cite[Corollary~2.2]{DuTang2013} on the set \(\Omega_m''\). We also recall that \(u^m \in L^2_{\cF^W}([0,T];W^{2,2}(\R))\), \(v^m \in L^2_{\cF^W}([0,T];W^{1,2}(\R))\), \(F_m \in L^2_{\cF^W}([0,T];L^2(\R))\) and \(\overline{Y}\) is a strong solution and therefore a continuous semimartingale.

  Moreover, we note that \(\rho,\frac{\dd}{\dd x} u\) are \(\mathcal{P}^W \times \mathcal{B}(\R)\)-measurable and \(\tau_m\) is \(\mathcal{P}^W \)-measurable. Hence, the same holds true for \(F_m\). Also as previously mentioned \((k*\rho_t)\) is bounded in \(x \in \R\) for almost all \((\omega, t ) \in \Omega \times [0,T_1]\) and as we have seen in the proof of Theorem~\ref{theorem: existence_mean_field_SDE} (Step~1) is Lipschitz continuous for almost all \((\omega, t ) \in \Omega \times [0,T_1]\). Thus, all assumptions of \cite[Theorem~3.1]{Yang2013DynkinGO} are fulfilled and we obtain
  \begin{align}\label{eq: ito_ventzel_formula_bspde}
    &u^m_{T_1}(\overline{Y}_{T_1}) \nonumber  \\
    &\quad=  u^m_0 (\overline{Y}_0) + \int\limits_{0}^{T_1} \bigg( \frac{\sigma^2_t+\nu^2}{2} \frac{\dd^2}{\dd x^2} u^m_t - (k*\rho_t) \frac{\dd}{\dd x} u_t + \nu \frac{\dd}{\dd x} v^m_t -  \frac{\sigma^2_t+\nu^2}{2} \frac{\dd^2}{\dd x^2} u^m_t\nonumber \\
    &\quad\quad+ \indicator{(0,\tau_m]}(t) (k*\rho_t) \frac{\dd}{\dd x} u_t  - \nu \frac{\dd}{\dd x} v^m_t \bigg)(\overline{Y}_t) \Id t   \nonumber   \\
    &\quad\quad+  \int\limits_{0}^{T_1} \bigg( v^m_t +  \nu \frac{\dd}{\dd x} u^m_t  \bigg)(\overline{Y}_t) \Id W_t +   \int\limits_{0}^{T_1} \sigma_t(\overline{Y}_t) \frac{\dd}{\dd x} u^m_t (\overline{Y}_t) \Id B_t \nonumber  \\
    &\quad= u^m_0 (\overline{Y}_0) + \int\limits_{0}^{T_1} F(t,\overline{Y}_t)(\indicator{(0,\tau_m]}(t)-1) \Id t
    +\int\limits_{0}^{T_1} \bigg( v^m_t +  \nu \frac{\dd}{\dd x} u^m_t  \bigg)(\overline{Y}_t) \Id W_t \nonumber\\
    &\quad\quad+  \int\limits_{0}^{T_1} \sigma_t(\overline{Y}_t) \frac{\dd}{\dd x} u^m_t (\overline{Y}_t) \Id B_t.
  \end{align}
  With this formula at hand, let us introduce the filtration \(\mathcal{G}_t=\sigma(\sigma(\overline{Y}_t),\cF_t^W)\), \(t \in [0,T]\). Our aim is to take the conditional expectation with respect to \(\mathcal{G}_0\) on both sides of the above equation in order to cancel both the stochastic integrals. We observe that \(\mathcal{G}_t \subset \cF_t\) and the solution \((\overline{Y}_t, t \in [0,T])\) is predictable with respect to the filtration \(\mathcal{G}\). Moreover, \(B^1\) and \(W\) are still per definition Brownian motions under the filtration \((\cF_t, t \in [0,T])\). Hence, both stochastic integrals are martingales with respect to the filtration \((\cF_t , t\in [0,T])\), if we can prove an \(L^2\)-bound on the integrands.

  By Sobolev's embedding or Morrey's inequality and the bound on \(\sigma_t\) we have
  \begin{align*}
    \E \bigg( \int\limits_{0}^{T_1} \bigg| \sigma_t(\overline{Y}_t) \frac{\dd}{\dd x} u^m_t(\overline{Y}_t) \bigg|^2 \Id t  \bigg)
    &\le \Lambda  \E\bigg( \int\limits_0^{T_1}   \norm{\frac{\dd}{\dd x} u^m_t}_{L^\infty(\R)}^2 \Id t \bigg) \\
    &\le \Lambda  \E\bigg( \int\limits_0^{T_1}   \norm{ u^m_t}_{W^{2,2}(\R)}^2 \Id t \bigg) < \infty,
  \end{align*}
  which verifies that the second stochastic integral of \eqref{eq: ito_ventzel_formula_bspde} is a martingale with respect to the filtration~\(\cF\). Hence, we discover
  \begin{align*}
    \E \bigg( \int\limits_{0}^{T_1} \sigma_t(\overline{Y}_t) \frac{\dd}{\dd x} u^m_t (\overline{Y}_t) \Id B_t \,\bigg |\, \mathcal{G}_0 \bigg)
    = \E \bigg(  \E\bigg( \int\limits_{0}^{T_1} \frac{\dd}{\dd x} u^m_t (\overline{Y}_t) \Id B_t \, | \cF_0 \bigg) \,\bigg |\, \mathcal{G}_0 \bigg)
    =0.
  \end{align*}
  Furthermore, we have the estimate
  \begin{align*}
    \E\bigg( \int\limits_0^{T_1} \bigg|v^m_t (\overline{Y}_t) +  \nu \frac{\dd}{\dd x} u^m_t (\overline{Y}_t)\bigg|^2 \Id t \bigg)
    &\le 2 \E\bigg( \int\limits_0^{T_1}    \norm{v^m_t}_{L^\infty(\R)}^2 + \nu^2  \norm{\frac{\dd}{\dd x} u^m_t}_{L^\infty(\R)}^2       \Id t  \bigg)  \\
    &\le C  \E\bigg( \int\limits_0^{T_1}   \norm{v^m_t}_{W^{1,2} (\R)}^2 +  \norm{u^m_t}_{W^{2,2} (\R)}^2    \Id t  \bigg) \\
    &<\infty,
  \end{align*}
  where we used Morrey's inequality in the second step. Hence, the first stochastic integral appearing in \eqref{eq: ito_ventzel_formula_bspde} is also a martingale with respect to the filtration \(\mathcal{G}\) starting at zero. Taking the conditional expectation with respect to \(\mathcal{G}_0\) in \eqref{eq: ito_ventzel_formula_bspde} and having in mind that \(\overline{Y}_0 = Y_0\), we obtain
  \begin{equation*}
    \E( u^m_{T_1}(\overline{Y}_{T_1}) \, | \, \mathcal{G}_0) = u^m_{0}(Y_{0}) +  \E\bigg( \int\limits_{0}^{T_1} F(t,\overline{Y}_t)(\indicator{(0,\tau_m]}(t)-1) \Id t \, \bigg| \, \mathcal{G}_0 \bigg).
  \end{equation*}

  It remains to show that
  \begin{align}\label{eq: convergence_bspde_cond_expectation}
    \lim\limits_{m \to \infty} ( \E( u^m_{T_1}(\overline{Y}_{T_1}) \, | \, \mathcal{G}_0) -u^m_{0}(Y_{0}))
    = \E ( u_{T_1}(\overline{Y}_{T_1}) \, | \, \mathcal{G}_0) -  u_{0}(Y_{0}),  \quad   \P\textnormal{-}\text{a.e.},
  \end{align}
  and
  \begin{equation}\label{eq: vanish_bspde_cond_expectation}
    \lim\limits_{m \to \infty} \E\bigg( \int\limits_{0}^{T_1} F(t,\overline{Y}_t)(\indicator{(0,\tau_m]}(t)-1) \Id t \,\bigg  | \, \mathcal{G}_0 \bigg) = 0, \quad   \P\textnormal{-}\text{a.e.}
  \end{equation}
  We first show \eqref{eq: vanish_bspde_cond_expectation} but we prove the \(L^1\)-convergence, which then implies \eqref{eq: vanish_bspde_cond_expectation} along a subsequence. We compute
  \begin{align*}
    &\E \bigg(\bigg|  \E\bigg( \int\limits_{0}^{T_1} F(t,\overline{Y}_t)(\indicator{(0,\tau_m]}(t)-1) \Id t \, \bigg| \, \mathcal{G}_0  \bigg)   \bigg| \bigg)  \\
    &\quad\le \, \E \bigg( \E\bigg( \bigg|\int\limits_{0}^{T_1} F(t,\overline{Y}_t)(\indicator{(0,\tau_m]}(t)-1) \Id t \bigg| \,\bigg  | \, \mathcal{G}_0  \bigg)  \bigg)  \\
    &\quad\le \,  \E \bigg( \int\limits_{0}^{T_1}| F(t,\overline{Y}_t)(\indicator{(0,\tau_m]}(t)-1)| \Id t \bigg) \\
    &\quad\le \,  \E \bigg( \int\limits_{0}^{T_1} \norm{ F(t,\cdot)}_{L^\infty(\R)} |\indicator{(0,\tau_m]}(t)-1| \Id t \bigg)  \\
    &\quad\le \,  \E \bigg( \int\limits_{0}^{T_1} \norm{ F(t,\cdot)}_{W^{1,2}(\R)} |\indicator{(0,\tau_m]}(t)-1| \Id t \bigg)  \\
    &\quad\le \,  C \E \bigg( \int\limits_{0}^{T_1}\bigg( \norm{ (k*\rho_t) \frac{\dd}{\dd x} u_t}_{L^2(\R)} +\norm{ (k*\rho_t) \frac{\dd^2}{\dd x^2} u_t}_{L^2(\R)}\\
    &\quad\quad+ \norm{ (k* \frac{\dd}{\dd x}\rho_t) \frac{\dd}{\dd x} u_t}_{L^2(\R)} \bigg) |\indicator{(0,\tau_m]}(t)-1| \Id t \bigg)  \\
    &\quad\le \, C(T) \E \bigg( \int\limits_{0}^{T_1} \bigg(2 \norm{\rho_t}_{L^2(\R)} + \norm{ \frac{\dd}{\dd x}\rho_t}_{L^2(\R)}  \bigg) \norm{u_t}_{W^{2,2}(\R)} |\indicator{(0,\tau_m]}(t)-1| \Id t \bigg) \\
    &\quad\le \,  C(T) \E \bigg( \int\limits_{0}^{T_1}  (\norm{\rho_t}_{W^{1,2}(\R)}^2 + \norm{u_t}_{W^{2,2}(\R)}^2 ) |\indicator{(0,\tau_m]}(t)-1| \Id t \bigg)  ,
  \end{align*}
  where we used Morrey's inequality in the fourth step and H{\"o}lder's inequality as well as \eqref{eq: l_infinity_bound_k_zeta} in the sixth step. Finally, \(\rho \in L^2_{\cF^W}([0,T];W^{1,2}(\R)), \; u \in  L^2_{\cF^W}([0,T];W^{2,2}(\R))\), dominated convergence theorem and \(\tau_m \uparrow T_1\) tell us that the last term vanishes for \(m \to \infty\).

  Taking the above subsequence, which we do not rename, we demonstrate \eqref{eq: convergence_bspde_cond_expectation} along a further subsequence by proving \(L^2\)-convergence of~\eqref{eq: convergence_bspde_cond_expectation}. Let us define \(\widetilde{u}^m = u-u^m\) and \(\widetilde{v}^m= v-v^m\), which solve the following BSPDE
  \begin{equation*}
    \Id \widetilde{u}^m_t = - \bigg( \frac{\sigma^2_t+\nu^2}{2} \frac{\dd^2}{\dd x^2} \widetilde{u}^m_t - \widetilde{F}_m  + \nu \frac{\dd}{\dd x} \widetilde{v}^m_t \bigg)  \Id t   + \widetilde{v}^m_t \Id W_t
  \end{equation*}
  with terminal condition \(\widetilde{G}=0\), free term
  \begin{equation*}
    \widetilde{F}_m(t.x) = (k*\rho)(t,x) \frac{\dd}{\dd x}u(t,x) (1-\indicator{(0,\tau_m]}(t)) = F(t,x)(1-\indicator{(0,\tau_m]}(t))
  \end{equation*}
  and \(\widetilde{F}_m \in  L^2_{\cF^W}([0,T];L^{2}(\R))\). Hence, by \cite[Theorem~5.5]{kai_du_L_p_bspde_2011} the solution of the BSPDE is unique and by \cite[Proposition~3.2 and Proposition~4.1, Step~1]{DU2BSPDE2010} satisfies the estimate
  \begin{equation}\label{eq: utilde_bspde_L2_estimate}
    \E\bigg( \sup\limits_{t \le T_1} \norm{\widetilde{u}^m_{t}}_{W^{1,2}(\R)}^2 \bigg)
    \le C(T) \norm{\widetilde{F}_m}_{L^{2}_{\cF^W}([0,T] ;L^{2}(\R))}.
  \end{equation}
  Consequently, using Jensen inequality, the \(\mathcal{G}_0\)-measurability of \(u_0(Y_0)\), Morrey's inequality and \eqref{eq: utilde_bspde_L2_estimate} we find
  \begin{align*}
    &\E (| \E( u^m_{T_1}(\overline{Y}_{T_1}) \, | \, \mathcal{G}_0) -u^m_0(Y_0)) - ( \E ( u_{T_1}(\overline{Y}_{T_1}) \, | \, \mathcal{G}_0) -u_0(Y_0) ) |^2 ) \\
    &\quad\le 2 \E ( |\widetilde{u}^m_{T_1}(\overline{Y}_{T_1})|^2 +| \widetilde{u}^m_0 (Y_0) |^2 ) \\
    & \quad\le 2 \E ( \norm{\widetilde{u}^m_{T_1}}_{L^\infty(\R)}^2 +  \norm{\widetilde{u}^m_{0}}_{L^\infty(\R)}^2 ) \\
    & \quad \le 2 \E( \norm{\widetilde{u}^m_{T_1}}_{W^{1,2}(\R)}^2 + \norm{\widetilde{u}^m_{0}}_{W^{1,2}(\R)}^2 )  \\
    &\quad\le 4 \E\bigg( \sup\limits_{t \le T_1} \norm{\widetilde{u}^m_{t}}_{W^{1,2}(\R)}^2 \bigg) \\
    &\quad\le C(T)  \E \bigg( \int\limits_{0}^{T_1}  \norm{(\widetilde{F}_m)_t}_{L^{2}(\R)}^2  \Id t\bigg) \\
    &\quad\le C(T)  \E\bigg( \int\limits_{0}^{T_1} |1-\indicator{(0,\tau_m]}(t)|^2
    \norm{ F(t,x)}_{L^{2}(\R)}^2  \Id t \bigg) .
  \end{align*}
  But \(F \in L^2_{\cF^W}([0,T] ;L^2(\R))\) and, therefore, an application of the dominated convergence theorem proves \eqref{eq: convergence_bspde_cond_expectation} along a subsequence. As a result, the last inequality together with \eqref{eq: vanish_bspde_cond_expectation} implies  \eqref{eq: bspde_charc_conditional_expectation} for all \(\omega \in \widetilde{\Omega}:= \bigcap\limits_{m \in \N} \Omega_m'' \cap  \Omega' \). Hence, the lemma is proven.
\end{proof}

As an application of  Theorem~\eqref{theorem: existence_mean_field_SDE}, we obtain a solution for the  non-regularized and the regularized mean-field SDEs~\eqref{eq: mean_field_trajectories_with_common_noise} and ~\eqref{eq: regularized_mean_field_trajectories_with_common_noise}, respectively.

\begin{corollary}\label{cor: solution to regularized mean-field SDE}
  Let \(0 \le \rho_0 \in L^1(\R) \cap W^{2,2}(\R) \) with \(\norm{\rho_0}_{L^1(\R)}=1\). Suppose that for \(T>0\) there exist unique solutions~$\rho$ and $\rho^{\tau}$ in \(L^2_{\cF^W}([0,T];W^{1,2}(\R))\) of the SPDEs~\eqref{eq: hk_spde} and~\eqref{eq: regularised_hk_spde}, respectively, on the interval \([0,T]\) with
  \begin{align*}
    &\norm{\rho}_{L^2_{\cF^w}([0,T];W^{1,2}(\R))} + \norm{\rho}_{S^\infty_{\cF^w}([0,T];L^1(\R) \cap L^2(\R))} \le C \quad \text{and}\\
    &\norm{\rho^{\tau}}_{L^2_{\cF^w}([0,T];W^{1,2}(\R))} + \norm{\rho^{\tau}}_{S^\infty_{\cF^w}([0,T];L^1(\R) \cap L^2(\R))} \le C,
  \end{align*}
  for some finite constant \(C>0\). Moreover, let the diffusion coefficient \(\sigma \colon [0,T] \times \R \to \R \) satisfy Assumption~\ref{ass: sigma_ellipticity+bound}. Then, for \(\tau >0\) there exists unique solution $(Y^{i}, t\in [0,T])$ and $(Y^{i,\tau}, t\in [0,T])$ for the mean-field SDEs~\eqref{eq: mean_field_trajectories_with_common_noise} and~\eqref{eq: regularized_mean_field_trajectories_with_common_noise}, respectively. Moreover, \(\rho_t\) is the conditional density of \(Y_t^{i}\) given \(\cF_t^W\) and \(\rho^\tau_t\) is the conditional density of \(Y_t^{i,\tau}\) given \(\cF_t^W\), for every $t\in [0,T]$ and for all \(i \in \N\).
\end{corollary}


\section{Mean-field limits of the interacting particle systems}\label{sec:mean field limit}

In this section we establish propagation of chaos for the regularized Hegselsmann--Krause models (in particular we recall \(k_{\scriptscriptstyle{HK}}^\tau(x)\) is the approximation of \(k_{\scriptscriptstyle{HK}}(x) =  \indicator{[-R,R]}(x) x \)) with environmental noise~\eqref{eq: regularized_particle_system_with_common_noise} towards the (non-regularized) mean-field stochastic differential equations~\eqref{eq: mean_field_trajectories_with_common_noise} resp. the (non-regularized) stochastic Fokker--Planck equation~\eqref{eq: hk_spde}, presenting the density based model of the opinion dynamics, see Theorem~\ref{theorem: propagation_of_chaos_with_noise}. To prove propagation of chaos, we first derive estimates of the difference of the regularized interacting particle system~\eqref{eq: regularized_particle_system_with_common_noise} and the regularized mean-field SDE~\eqref{eq: regularized_mean_field_trajectories_with_common_noise} (see Proposition~\ref{prop: l_2_estimate_real_approx_trajectory_mean_field_trajectory}) as well as of the difference of the solutions to the regularized stochastic Fokker--Planck equation~\eqref{eq: regularised_hk_spde} and of the non-regularized stochastic Fokker--Planck equation~\eqref{eq: hk_spde} (see Proposition~\ref{prop: difference of regularized mean-field SDE and the mean-field SDE}). As preparation, we need the following auxiliary lemma.

\begin{lemma}\label{lemma: explicit_conditional_expectation}
  Let \((\Omega, \mathcal{F}, \P)\) be a probability space, \(\mathcal{G} \subseteq \mathcal{F}\) a sub-\(\sigma\)-algebra and \(X,Y\) conditionally independent random variables with values in \(\R\) given \(\mathcal{G}\). Moreover, let \(X\) have a conditional density \(f \colon \Omega \times \R \to \R\) such that \(f\) is \(\mathcal{G} \otimes \mathcal{B}(\R) / \mathcal{B}(\R)  \) measurable and in \(L^1(\Omega \times \R)\). Then, for every bounded measurable function \(h\colon \R \times \R \to \R\), we have
  \begin{equation}\label{eq: explicit_representation_cond_expectation}
    \E( h(X,Y) \,| \, \sigma(\mathcal{G}, \sigma(Y) ))(\omega) = \int_{\R} h(z,Y(\omega)) f(\omega,z) \Id z,
    \quad \omega \in \Omega^\prime,
  \end{equation}
  on a set $\Omega^\prime\subset \Omega$ of full probability.
\end{lemma}

\begin{proof}
  First, we notice that by Fubini's theorem the right-hand side of \eqref{eq: explicit_representation_cond_expectation} is \(\sigma(\mathcal{G},\sigma(Y))\)-measurable. By the standard Lebesgue integral approximation technique we may assume \(h= \indicator{B \times B'}(x,y)\) for some measurable sets \(B, B' \in \mathcal{B}(\R)\) in order to prove \eqref{eq: explicit_representation_cond_expectation}. Hence, we need to show
  \begin{equation*}
    \E(\indicator{A} \indicator{B \times B'}(X,Y)) = \E \bigg( \indicator{A} \int_{\R} \indicator{B \times B'}(z,Y(\omega)) f(\omega,z) \Id z  \bigg)
  \end{equation*}
  for all \(A \in \sigma(\mathcal{G},\sigma(Y))\). Now, we reduce the problem again to \(A = C \cap C''\) with \(C \in \mathcal{G} \) and \(C'' = Y^{-1}(B'')\) for some \(B'' \in \mathcal{B}(\R)\). Consequently, using the conditional independence we find
  \begin{align*}
    \E(\indicator{C \cap C''} \indicator{B \times B'}(X,Y))
    &= \E(\indicator{C} \E(\indicator{C''} \indicator{B \times B'}(X,Y) \, | \, \mathcal{G} ) ) \\
    &= \E(\indicator{C} \E(\indicator{B'' \cap B'}(Y)  \indicator{B}(X) \, | \, \mathcal{G} ) ) \\
    &= \E(\indicator{C} \E(\indicator{B'' \cap B'}(Y) \, | \, \mathcal{G} ) \E(  \indicator{B}(X) \, | \, \mathcal{G} ) ) \\
    &= \E(\indicator{C} \indicator{B'' \cap B'}(Y) \E(  \indicator{B}(X) \, | \, \mathcal{G} ) ) \\
    &= \E \bigg(\indicator{C \cap C''}(\omega) \indicator{B'}(Y(\omega)) \int_{\R} \indicator{B}(z) f(\omega,z) \Id z   \bigg)\\
    &= \E \bigg(\indicator{C \cap C''}(\omega) \int_{\R} \indicator{B \times B'}(z,Y(\omega)) f(\omega,z) \Id z \bigg)
  \end{align*}
  and the lemma is proven.
\end{proof}

The next proposition provides an estimate of the difference of the regularized particle system and the regularized mean-field SDE.

\begin{proposition}\label{prop: l_2_estimate_real_approx_trajectory_mean_field_trajectory}
  Suppose Assumption~\ref{ass: sigma_ellipticity+bound} and Assumption~\ref{ass: spde_existence} hold. For each \(N \in \N\), let \(((Y_t^{i,\tau}, t \in [0,T]), i=1,\ldots,N)\) be the solutions to the regularized mean-field SDEs~\eqref{eq: regularized_mean_field_trajectories_with_common_noise}, as provided by Corollary~\ref{cor: solution to regularized mean-field SDE}, and let \(((X_t^{i,\tau}, t \in [0,T]), i=1,\ldots,N)\) be the solution to regularized interaction particle system~\eqref{eq: regularized_particle_system_with_common_noise}. Then, for any \(\tau > 0\) and \(N\in \N\) we have
  \begin{equation*}
    \sup\limits_{t \in [0,T]} \sup\limits_{i=1,\ldots,N}  \E(|X_t^{i,\tau} - Y_t^{i,\tau} |^2 )
    \le  \frac{ 2 \norm{k_{\scriptscriptstyle{HK}}}_{L^2(\R)}^2 T}{(N-1) \tau} \exp\bigg( \frac{(C+\Lambda) T }{\tau} \bigg),
  \end{equation*}
  where \(C\) is some finite generic constant.  
\end{proposition}

\begin{proof}
  Applying It{\^o}'s formula, we find
  \begin{align}\label{eq: regularised_particle_regularised_mean_field_l_2_norm}
    &|X_t^{i,\tau} - Y_t^{i,\tau} |^2 \nonumber\\
    &\quad=  \; 2 \int\limits_0^t (X_s^{i,\tau} - Y_s^{i,\tau} ) \bigg( \frac{1}{N-1} \sum\limits_{\substack{j=1 \\j \neq i }}^N -k_{\scriptscriptstyle{HK}}^\tau(X_s^{i,\tau} -X_s^{j,\tau})  +(k_{\scriptscriptstyle{HK}}^\tau*\rho_s^\tau) (Y_s^{i,\tau}) \bigg)  \Id s \nonumber \\
    &\quad\quad+ 2 \int\limits_0^t (X_s^{i,\tau} - Y_s^{i,\tau} ) ( \sigma(s,X_s^{i,\tau}) - \sigma(s,Y_s^{i,\tau})) \Id B_s^{i}  + \int\limits_0^t ( \sigma(s,X_s^{i,\tau}) - \sigma(s,Y_s^{i,\tau}))^2 \Id s .
  \end{align}
  Splitting the sum we have
  \begin{align*} 
    &\frac{1}{N-1} \sum\limits_{\substack{j=1 \\j \neq i }}^N -k_{\scriptscriptstyle{HK}}^\tau(X_s^{i,\tau} -X_s^{j,\tau})  + (k_{\scriptscriptstyle{HK}}^\tau*\rho_s^\tau) (Y_s^{i,\tau}) \\
    &\quad= \,  \frac{1}{N-1} \sum\limits_{\substack{j=1 \\j \neq i }}^N (k_{\scriptscriptstyle{HK}}^\tau*\rho_s^\tau) (Y_s^{i,\tau})  - k_{\scriptscriptstyle{HK}}^\tau(Y_s^{i,\tau}-Y_s^{j,\tau})  \\
    &\quad\quad+ \frac{1}{N-1} \sum\limits_{\substack{j=1 \\j \neq i }}^N k_{\scriptscriptstyle{HK}}^\tau(Y_s^{i,\tau}-Y_s^{j,\tau}) -k_{\scriptscriptstyle{HK}}^\tau(X_s^{i,\tau} -X_s^{j,\tau}) \\
    &\quad= \, I_1^s + I_2^s .
  \end{align*}
  For \(I_2^s \), we use the property of our approximation sequence to discover
  \begin{align*}
    |I^s_2|
    &\le  \frac{1}{N-1} \sum\limits_{\substack{j=1 \\j \neq i }}^N |k_{\scriptscriptstyle{HK}}^\tau(Y_s^{i,\tau}-Y_s^{j,\tau}) -k_{\scriptscriptstyle{HK}}^\tau(X_s^{i,\tau} -X_s^{j,\tau})|\\ &\le \frac{C}{(N-1)\tau} \sum\limits_{\substack{j=1}}^N  |X_s^{j,\tau} - Y_s^{j,\tau}| +  |X_s^{i,\tau} - Y_s^{i,\tau}|
  \end{align*}
  and consequently
  \begin{align*}
    \E \bigg( \bigg| 2 \int\limits_0^t (X_s^{i,\tau} - Y_s^{i,\tau} ) I_2^s \Id s \bigg| \bigg)
    &\le \frac{C}{(N-1)\tau} \int\limits_0^t   \sum\limits_{\substack{j=1}}^N  \E( |X_s^{j,\tau} - Y_s^{j,\tau} |^2 +  |X_s^{i,\tau} - Y_s^{i,\tau} |^2 )  \Id s \\
    &\le \frac{C}{\tau} \int\limits_0^t  \sup\limits_{i=1,\ldots,N}  \E( |X_s^{i,\tau} - Y_s^{i,\tau} |^2  )  \Id s  ,
  \end{align*}
  where we used Young's inequality. Next, let us rewrite \(I_1^s\) such that
  \begin{align*}
    I_1^s  = \frac{1}{N-1} \sum\limits_{\substack{j=1 \\j \neq i }}^N (k_{\scriptscriptstyle{HK}}^\tau*\rho_s^\tau) (Y_s^{i,\tau})  - k_{\scriptscriptstyle{HK}}^\tau(Y_s^{i,\tau}-Y_s^{j,\tau}) = \frac{1}{N-1} \sum\limits_{\substack{j=1 \\j \neq i }}^N Z_{i,j}^s
  \end{align*}
  with
  \begin{equation*}
    Z_{i,j}^s = (k_{\scriptscriptstyle{HK}}^\tau*\rho_s^\tau) (Y_s^{i,\tau})  - k_{\scriptscriptstyle{HK}}^\tau(Y_s^{i,\tau}-Y_s^{j,\tau})
  \end{equation*}
  for \( i \neq j \). Furthermore,
  \begin{align*}
    \E(|I_1^s|^2)
    &= \frac{1}{(N-1)^2} \E\bigg( \E \bigg( \sum\limits_{\substack{j=1 \\j \neq i }}^N Z_{i,j}^s \sum\limits_{\substack{k=1 \\j \neq i }}^N Z_{i,k}^s  \, \bigg| \, \sigma( \cF_s^W,  \sigma(Y_s^{i,\tau} )) \bigg) \bigg)  \\
    &= \frac{1}{(N-1)^2} \sum\limits_{\substack{j=1 \\j \neq i }}^N \sum\limits_{\substack{k=1 \\j \neq i }}^N \E( \E(Z_{i,j}^s Z_{i,k}^s \, | \, \sigma( \cF_s^W, \sigma(Y_s^{i,\tau} ) ) ) .
  \end{align*}
  It easy to verify that \((Y_s^{i,\tau},i=1,\ldots,N)\) are conditionally independent given \(\cF_s^W\) and by Theorem~\ref{theorem: existence_mean_field_SDE} have conditional density \(\rho_s\). Hence, we apply Lemma~\ref{lemma: explicit_conditional_expectation} to obtain
  \begin{equation*}
    \E(k_{\scriptscriptstyle{HK}}^\tau(Y_s^{i,\tau}-Y_s^{j,\tau}) \, | \, \sigma( \cF_s^W , \sigma( Y_s^{i},\tau)  ) ) = (k_{\scriptscriptstyle{HK}}^\tau*\rho_s^\tau) (Y_s^{i,\tau})
  \end{equation*}
  and therefore \(\E(Z_{i,j}^s \, | \, \sigma(\cF_s^W , \sigma( Y_s^{i,\tau}))) =0\) since \((k_{\scriptscriptstyle{HK}}^\tau*\rho_s^\tau) (Y_s^{i,\tau})\) is \(\sigma(\cF_s^W,  \sigma(Y_s^{i,\tau}))\)-measurable. Consequently, for the cross terms \(j \neq k \) one can verify that
  \begin{equation*}
    \E(Z_{i,j}^s Z_{i,k}^s \, | \, \sigma( \cF_s^W, \sigma( Y_s^{,\tau} ))  ) = \E(Z_{i,j}^s \, | \, \sigma( \cF_s^W, \sigma( Y_s^{,\tau} ))   ) \E(Z_{i,k}^s \, | \, \sigma( \cF_s^W, \sigma( Y_s^{,\tau} ))  )  =0
  \end{equation*}
  by the previous findings. Hence, we have
  \begin{equation*}
    \E(|I_1^s|^2)
    = \frac{1}{(N-1)^2} \sum\limits_{\substack{j=1 \\j \neq i }}^N \E(|Z_{i,j}^s|^2 )
  \end{equation*}
  and using the boundedness of \(k^\tau\), the structure of our approximation and mass conservation, we obtain
  \begin{equation*}
    \E(|Z_{i,j}^s|^2 )
    = \E ( | (k_{\scriptscriptstyle{HK}}^\tau*\rho_s^\tau) (Y_s^{i,\tau})
    - k_{\scriptscriptstyle{HK}}^\tau(Y_s^{i,\tau}-Y_s^{j,\tau})|^2)
    \le 2 \norm{k_{\scriptscriptstyle{HK}}^\tau}_{L^\infty(\R)}^2
    \le \frac{2}{\tau} \norm{k_{\scriptscriptstyle{HK}}}_{L^2(\R)}^2.
  \end{equation*}
  Combining all the above facts, we get
  \begin{equation*}
    \E(|I_1^s|^2) \le \frac{ 2 \norm{k_{\scriptscriptstyle{HK}}}_{L^2(\R)}^2}{(N-1)\tau}
  \end{equation*}
  and
  \begin{align*}
    \E\bigg( 2 \int\limits_0^t (X_t^{i,\tau} - Y_t^{i,\tau} ) I_1^s \Id s \bigg)
    &\le  \E \bigg(\int\limits_0^t |X_t^{i,\tau} - Y_t^{i,\tau} |^2 \Id s  + \int\limits_0^t | I_1^s|^2 \Id s \bigg) \\
    &\le  \int\limits_0^t \E(|X_t^{i,\tau} - Y_t^{i,\tau} |^2 ) \Id s + \frac{ 2 \norm{k_{\scriptscriptstyle{HK}}}_{L^2(\R)}^2T}{(N-1)\tau}.
  \end{align*}
  Moreover, using the Lipschitz continuity of our coefficients \(\sigma\) we obtain
  \begin{equation*}
    \E \Bigg( \int\limits_0^t ( \sigma(s,X_s^{i,\tau}) - \sigma(s,Y_s^{i,\tau}))^2  \Id s \Bigg)
    \le \Lambda  \int\limits_0^t \E \big(  |X_s^{i,\tau}-Y_s^{i,\tau}|^2 \big)
    \le \Lambda \int\limits_0^t \sup\limits_{i=1,\ldots , N}  \E \big(  |X_s^{i,\tau}-Y_s^{i,\tau}|^2 \big).
  \end{equation*}
  Now, combining this with the estimate of \(I^s_2\), as well as the fact that the stochastic integral in equation ~\eqref{eq: regularised_particle_regularised_mean_field_l_2_norm} is a martingale (Assumption~\ref{ass: sigma_ellipticity+bound}), we obtain
  \begin{align*}
    \sup\limits_{i=1,\ldots,N} \E(|X_t^{i,\tau} - Y_t^{i,\tau} |^2 )
    \le & \; \frac{C}{\tau} \int\limits_0^t \sup\limits_{i=1,\ldots,N} \E (  |X_t^{i,\tau} - Y_t^{i,\tau} |^2 ) \Id s   \\
    &+ \Lambda  \int\limits_0^t \sup\limits_{i=1,\ldots,N} \E(|X_t^{i,\tau} - Y_t^{i,\tau} |^2 ) \Id s +  \frac{ 2 \norm{k_{\scriptscriptstyle{HK}}}_{L^2(\R)}^2 T}{(N-1)\tau} \\
    \le & \; \frac{C+ \Lambda}{\tau} \int\limits_0^t \sup\limits_{i=1,\ldots,N} \E( |X_t^{i,\tau} - Y_t^{i,\tau} |^2) \Id s
    + \frac{ 2 \norm{k_{\scriptscriptstyle{HK}}}_{L^2(\R)}^2 T}{(N-1)\tau}.
  \end{align*}
  Applying Gronwall's inequality yields
  \begin{equation*}
    \sup\limits_{t \in [0,T]} \sup\limits_{i=1,\ldots,N} \E(|X_t^{i,\tau} - Y_t^{i,\tau} |^2 )
    \le   \frac{ 2 \norm{k_{\scriptscriptstyle{HK}}}_{L^2(\R)}^2 T}{(N-1)\tau} \exp\bigg( \frac{(C+\Lambda) T }{\tau} \bigg),
  \end{equation*}
  which proves the proposition.
\end{proof}

In the next step we need to estimate the difference of the solutions to the regularized mean-field SDEs and the non-regularized mean-field SDE. Recall that, by the stochastic Fokker--Planck equations, it is sufficient to consider the associated solutions \( \rho^\tau \) and \(\rho\) of the SPDEs~\eqref{eq: regularised_hk_spde} and \eqref{eq: hk_spde}. For more details regarding this observation we refer to the proof of Theorem~\ref{theorem: existence_mean_field_SDE}.

\begin{proposition}\label{prop: difference of regularized mean-field SDE and the mean-field SDE}
  Suppose Assumption~\ref{ass: sigma_ellipticity+bound} and Assumption~\ref{ass: spde_existence} hold. Let $\rho^\tau$ and $\rho$  be the solutions to the regularized stochastic Fokker--Planck equation~\eqref{eq: regularised_hk_spde} and to the non-regularized stochastic Fokker--Planck equation~\eqref{eq: hk_spde}, respectively, as provided in Corollary~\ref{cor: hk_spde_well_pos}. Then,
  \begin{align*}
    &  \norm{\rho^\tau - \rho}_{S^\infty_{\cF^W}([0,T];L^2(\R))} \\
    &\quad \le  \;   C( \lambda, \Lambda,  T, \norm{k_{\scriptscriptstyle{HK}}}_{L^2(\R)},  \norm{\rho}_{S^\infty_{\cF^W}([0,T];L^2(\R))} \norm{\rho^\tau}_{S^\infty_{\cF^W}([0,T];L^2(\R))} )  \norm{k_{\scriptscriptstyle{HK}}^\tau-k_{\scriptscriptstyle{HK}}}_{L^2(\R)} .
  \end{align*}
\end{proposition}

\begin{proof}
  To estimate the difference \(\rho_t - \rho_t^\tau\), we notice that
  \begin{align*}
    \rho_t^\tau- \rho_t
    &=  \frac{\dd^2}{\dd x^2} \Bigg( \frac{\sigma^2_t + \nu^2 }{2}  (\rho_t^\tau-\rho_t)\Bigg)  \Id t + \frac{\dd}{\dd x} ((k_{\scriptscriptstyle{HK}}^\tau *\rho_t^\tau)\rho_t^\tau)  \Id t \\
    &\quad- \frac{\dd}{\dd x} ((k_{\scriptscriptstyle{HK}} *\rho_t)\rho_t)  \Id t - \nu \frac{\dd}{\dd x} ( \rho_t^\tau- \rho_t) \Id W_t .
  \end{align*}
  Performing similar computations as in the proof of Theorem~\ref{theorem: existence_of_hk_spde} by using Young's inequality, we get
  \begin{align*}
    &\norm{\rho_t^\tau - \rho_t}_{L^2(\R)}^2\\
    &\quad\le  \;  - \lambda \int\limits_0^t  \norm{\frac{\dd}{\dd x} \rho^\tau_s - \frac{\dd}{\dd x}\rho_s}_{L^2(\R)}^2   \Id s
    -  \int\limits_0^t \bigg \langle (\rho_s - \rho^\tau_s) \frac{\dd}{\dd x} (\sigma_s^2) ,  \frac{\dd}{\dd x} \rho_s - \frac{\dd}{\dd x} \rho^\tau_s \bigg \rangle_{L^2(\R)}  \Id s \\
    &\quad\quad-  2 \int\limits_0^t  \bigg \langle (k_{\scriptscriptstyle{HK}}^\tau *\rho_s^\tau) \rho_s^\tau  - (k_{\scriptscriptstyle{HK}}*\rho_s )\rho_s , \frac{\dd}{\dd x}  \rho^\tau_s - \frac{\dd}{\dd x} \rho_s \bigg \rangle_{L^2(\R)}  \Id s \\
    &\quad\le \;  - \frac{3\lambda}{4} \int\limits_0^t  \norm{\frac{\dd}{\dd x} \rho^\tau_s - \frac{\dd}{\dd x} \rho_s}_{L^2(\R)}^2   \Id s
    +  \frac{\Lambda^2}{\lambda}\int\limits_0^t \norm{\rho_s - \rho^\tau_s}_{L^2(\R)}^2   \Id s \\
    &\quad\quad-  2 \int\limits_0^t  \bigg \langle (k_{\scriptscriptstyle{HK}}^\tau *\rho_s^\tau) \rho_s^\tau  - (k_{\scriptscriptstyle{HK}} *\rho_s )\rho_s , \frac{\dd}{\dd x} \rho^\tau_s - \frac{\dd}{\dd x} \rho_s \bigg\rangle_{L^2(\R)}  \Id s .
  \end{align*}
  Rewriting the last term gives
  \begin{align*}
    &(k_{\scriptscriptstyle{HK}}^\tau *\rho_s^\tau) \rho_s^\tau  - (k_{\scriptscriptstyle{HK}} *\rho_s )\rho_s\\
    &\quad=  ((k_{\scriptscriptstyle{HK}}^\tau-k_{\scriptscriptstyle{HK}}) *\rho_s^\tau) \rho_s^\tau + (k_{\scriptscriptstyle{HK}}*\rho_s^\tau) \rho_s^\tau - (k_{\scriptscriptstyle{HK}} *\rho_s )\rho_s  \\
    &\quad=  ((k_{\scriptscriptstyle{HK}}^\tau-k_{\scriptscriptstyle{HK}}) *\rho_s^\tau) \rho_s^\tau + (k_{\scriptscriptstyle{HK}}*(\rho_s^\tau-\rho_s)) \rho_s^\tau + ((k_{\scriptscriptstyle{HK}}*\rho_s )(\rho_s^\tau-\rho_s).
  \end{align*}
  Hence, for the last two terms we can use Young's inequality, Young's inequality for convolution, mass conservation and \eqref{eq: l_infinity_bound_k_zeta} to obtain
  \begin{align*}
    &\bigg \langle (k_{\scriptscriptstyle{HK}} *\rho_s )(\rho_s^\tau- \rho_s ) + (k_{\scriptscriptstyle{HK}} *(\rho_s^\tau- \rho_s) )\rho_s^\tau  , \frac{\dd}{\dd x} \rho^\tau_s -\frac{\dd}{\dd x}  \rho_s \bigg\rangle_{L^2(\R)} \\
    &\quad\le \norm{k_{\scriptscriptstyle{HK}}}_{L^2(\R)} \norm{\rho_s}_{L^2(\R)} \norm{\rho_s^\tau -\rho_s}_{L^2(\R)}  \norm{ \frac{\dd}{\dd x} \rho_s^\tau - \frac{\dd}{\dd x} \rho_s }_{L^2(\R)}\\
    &\quad\quad+ \norm{k_{\scriptscriptstyle{HK}}}_{L^2(\R)}  \norm{\rho_s^\tau -\rho_s}_{L^2(\R)} \norm{\rho_s^\tau}_{L^2(\R)}  \norm{\frac{\dd}{\dd x} \rho_s^\tau -\frac{\dd}{\dd x} \rho_s }_{L^2(\R)}  \\
    &\quad\le \frac{\lambda}{4}  \norm{ \frac{\dd}{\dd x} \rho_s^\tau - \frac{\dd}{\dd x} \rho_s }_{L^2(\R)}^2  \\
    &\quad\quad+ \frac{1}{\lambda} ( \norm{k_{\scriptscriptstyle{HK}}}_{L^2(\R)} \norm{\rho_s}_{L^2(\R)} \norm{\rho_s^\tau -\rho_s}_{L^2(\R)}+  \norm{k_{\scriptscriptstyle{HK}}}_{L^2(\R)} \norm{\rho_s^\tau -\rho_s}_{L^2(\R)} \norm{\rho_s^\tau}_{L^2(\R)} )^2  \\
    &\quad\le  \frac{\lambda}{2}  \norm{\frac{\dd}{\dd x} \rho_s^\tau -\frac{\dd}{\dd x} \rho_s }_{L^2(\R)}^2 + \frac{2}{\lambda} \norm{k_{\scriptscriptstyle{HK}}}_{L^2(\R)}^2   \norm{\rho_s^\tau -\rho_s}_{L^2(\R)} ^2 (\norm{\rho_s}_{L^2(\R)}^2+\norm{\rho_s^\tau}_{L^2(\R)}^2).
  \end{align*}
  Moreover,
  \begin{align*}
    & \bigg \langle ((k_{\scriptscriptstyle{HK}}^\tau-k_{\scriptscriptstyle{HK}}) *\rho_s^\tau) \rho_s^\tau, \frac{\dd}{\dd x} \rho_s^\tau - \frac{\dd}{\dd x}  \rho_s \bigg \rangle_{L^2(\R)} \\
    &\quad \le \;  \norm{(k_{\scriptscriptstyle{HK}}^\tau-k_{\scriptscriptstyle{HK}})*\rho_s^\tau}_{L^\infty(\R)} \norm{\rho_s^\tau}_{L^2(\R)} \norm{\frac{\dd}{\dd x}  \rho^\tau_s -  \frac{\dd}{\dd x} \rho_s}_{L^2(\R)} \\ \
    &\quad \le \; \norm{k_{\scriptscriptstyle{HK}}^\tau-k_{\scriptscriptstyle{HK}}}_{L^2(\R)} \norm{\rho_s^\tau}_{L^2(\R)} \norm{\rho_s^\tau}_{L^2(\R)} \norm{\frac{\dd}{\dd x} \rho^\tau_s -  \frac{\dd}{\dd x} \rho_s}_{L^2(\R)} \\
    &\quad \le \; \frac{\lambda}{4} \norm{\frac{\dd}{\dd x} \rho^\tau_s -  \frac{\dd}{\dd x} \rho_s}_{L^2(\R)}^2 + \frac{1}{\lambda} \norm{k_{\scriptscriptstyle{HK}}^\tau-k_{\scriptscriptstyle{HK}}}_{L^2(\R)}^2 \norm{\rho_s^\tau}_{L^2(\R)}^4,
  \end{align*}
  where we used Young's inequality for convolutions in the second inequality and Young's inequality for the last step. Consequently, combining the last two estimates with our previous \(L^2\)-norm inequality and absorbing the \(L^2\)-norm of the derivatives we obtain
  \begin{align*}
    &\norm{\rho_t^\tau - \rho_t}_{L^2(\R)}^2\\
    &\quad\le \frac{2 \norm{k_{\scriptscriptstyle{HK}}}_{L^2(\R)}^2+1}{\lambda} \int\limits_0^t (\norm{\rho_s}_{L^2(\R)}^2 +\norm{\rho_s^\tau}_{L^2(\R)}^2 + \Lambda^2) \norm{\rho_s^\tau -\rho_s}_{L^2(\R)}^2 \Id s \\
    &\quad\quad+ \frac{1}{\lambda} \int\limits_0^t \norm{k_{\scriptscriptstyle{HK}}^\tau-k_{\scriptscriptstyle{HK}}}_{L^2(\R)}^2 \norm{\rho_s^\tau}_{L^2(\R)}^4 \Id s  \\
    &\quad \le \frac{2 \norm{k_{\scriptscriptstyle{HK}}}_{L^2(\R)}^2+1}{\lambda} \int\limits_0^t (\norm{\rho_s}_{L^2(\R)}^2 + \norm{\rho_s^\tau}_{L^2(\R)}^2+ \Lambda^2) \norm{\rho_s^\tau -\rho_s}_{L^2(\R)}^2 \Id s\\
    &\quad \quad + \frac{T}{\lambda}\norm{k_{\scriptscriptstyle{HK}}^\tau-k_{\scriptscriptstyle{HK}}}_{L^2(\R)}^2  \sup\limits_{t \in [0,T]} \norm{\rho_t^\tau}_{L^2(\R)}^4.
  \end{align*}
  Applying Gronwall's inequality and using the uniform bound~\eqref{eq: l_infinity_l_2_bound_of_rho}, we obtain
  \begin{align*}
    &\sup\limits_{t \in [0,T]}  \norm{\rho_t^\tau - \rho_t}_{L^2(\R)}^2 \\
    &\quad\le \frac{T}{\lambda}\norm{k_{\scriptscriptstyle{HK}}^\tau-k_{\scriptscriptstyle{HK}}}_{L^2(\R)}^2  \sup\limits_{t \in [0,T]} \norm{\rho_t^\tau}_{L^2(\R)}^4\\
    &\quad\quad\times \exp \bigg( \frac{2 \norm{k_{\scriptscriptstyle{HK}}}_{L^2(\R)}^2+1}{\lambda} \int\limits_0^T  (\norm{\rho_s}_{L^2(\R)}^2+\norm{\rho_s^\tau}_{L^2(\R)}^2 + \Lambda) \Id s  \bigg) \\
    &\quad\le \frac{T}{\lambda}\norm{k_{\scriptscriptstyle{HK}}^\tau-k_{\scriptscriptstyle{HK}}}_{L^2(\R)}^2  \sup\limits_{t \in [0,T]} \norm{\rho_t^\tau}_{L^2(\R)}^4\\
    &\quad\quad\times \exp \bigg( \frac{2T  (\norm{k_{\scriptscriptstyle{HK}}}_{L^2(\R)}^2+1)}{\lambda}  (\norm{\rho}_{S^\infty_{\cF^W}([0,T];L^2(\R))}^2 +\norm{\rho^\tau}_{S^\infty_{\cF^W}([0,T];L^2(\R))}^2 +  \Lambda) \bigg) .
  \end{align*}
  After taking the supremum over \(\omega \in \Omega\), we arrive at
  \begin{align*}
    &\norm{\rho^\tau - \rho}_{S^\infty_{\cF^W}([0,T];L^2(\R))}\\
    &\quad\le  C( \lambda, \Lambda,  T,\norm{k_{\scriptscriptstyle{HK}}}_{L^2(\R)},  \norm{\rho}_{S^\infty_{\cF^W}([0,T];L^2(\R))} \norm{\rho^\tau}_{S^\infty_{\cF^W}([0,T];L^2(\R))} )  \norm{k^\tau-k}_{L^2(\R)} .
  \end{align*}
\end{proof}

\begin{remark}
  Due to Proposition~\ref{prop: difference of regularized mean-field SDE and the mean-field SDE}, we know that the solutions~\(\rho^\tau\) of the regularized stochastic Fokker--Planck equations converges to the solution~\(\rho\) of the non-regularized stochastic Fokker--Planck equation as the interaction force kernels converge in the \(L^2\)-norm for \(\tau \to 0\).
\end{remark}

Finally, we are in a position to state and prove the main theorem of this section.

\begin{theorem}[Propagation of chaos]\label{theorem: propagation_of_chaos_with_noise}
  Suppose Assumption~\ref{ass: sigma_ellipticity+bound} and Assumption~\ref{ass: spde_existence} hold. Let $\rho$ be the solution of the stochastic Fokker--Planck equation~\eqref{eq: hk_spde} and let us denote by
  \begin{equation*}
  \Pi_t^{N,\tau}(\omega) :=\frac{1}{N} \sum\limits_{i=1}^N \delta_{X_t^{i,\tau}(\omega)}
  \end{equation*}
  the empirical measure of the regularized interaction system \(((X_t^{i,\tau}, \tau > 0), i =1 , \ldots, N )\) given by \eqref{eq: regularized_particle_system_with_common_noise}. Then, we have, for all \(t \in [0,T]\),
  \begin{align*}
    &\E(|\qv{\Pi_t^{N,\tau} , \varphi} -  \qv{ \rho_t , \varphi }|^2)\\
    &\quad \le \;  C \Big(\lambda, \Lambda,  T, \norm{k_{\scriptscriptstyle{HK}}}_{L^2(\R)}, \norm{\varphi}_{C^1(\R)}, \norm{\varphi}_{L^2(\R)}, \norm{\rho}_{S^\infty_{\cF^W}([0,T];L^2(\R))},  \norm{\rho^\tau}_{S^\infty_{\cF^W}([0,T];L^2(\R))} \Big)\\
    &\quad\quad \times \bigg(  \frac{1}{N} + \frac{1}{(N-1) \tau } \exp\bigg( \frac{(C+\Lambda) T }{\tau}  \bigg) +  \norm{k^\tau-k}_{L^2(\R)}^2  \bigg)
  \end{align*}
  for any \(\varphi \in \testfunctions{\R}\) and a finite constant \(C\). Consequently, we have
  \begin{equation*}
  \lim\limits_{N\to \infty} \E(|\qv{\Pi_t^{N,\log(N)}, \varphi} -  \qv{ \rho_t , \varphi }|^2) = 0.
\end{equation*}   
\end{theorem}

\begin{proof}
  We compute
  \begin{align}\label{eq: empirical_measure_estimate}
    &\E(|\qv{\Pi_t^{N,\tau} , \varphi}- \qv{\rho_t^{\tau}, \varphi}|^2) \nonumber  \\
    &\quad= \E \bigg( \bigg(\frac{1}{N} \sum\limits_{i=1}^N \varphi(X_t^{i,\tau}) - \int_{\R} \rho_t^\tau(x) \varphi(x) \Id x \bigg)^2 \bigg) \nonumber  \\
    &\quad=  \, 2\E \bigg(\bigg(\frac{1}{N} \sum\limits_{i=1}^N \varphi(X_t^{i,\tau}) - \frac{1}{N} \sum\limits_{i=1}^N \varphi(Y_t^{i,\tau})\bigg)^2 \bigg)\nonumber  \\
    &\quad\quad+ 2 \E\bigg( \bigg( \frac{1}{N} \sum\limits_{i=1}^N \varphi(Y_t^{i,\tau}) - \int_{\R} \rho_t^\tau(x) \varphi(x) \Id x \bigg)^2 \bigg) \nonumber  \\
    &\quad\le  \, \frac{2}{N^2}\bigg(\sum\limits_{i=1}^N  \E( |\varphi(X_t^{i,\tau}) - \varphi(Y_t^{i,\tau})|^2)^{1/2} \bigg)^2 \nonumber  \\
    &\quad\quad+  2 \E \bigg( \frac{1}{N} \sum\limits_{i=1}^N \varphi(Y_t^{i,\tau}) - \int_{\R} \rho_t^\tau(x) \varphi(x) \Id x \bigg)^2 \bigg)  \nonumber \\
    &\quad\le  \, 2 \sup\limits_{i=1,\ldots,N}  \E( |\varphi(X_t^{i,\tau}) - \varphi(Y_t^{i,\tau})|^2) +  2 \E \bigg(\bigg( \frac{1}{N} \sum\limits_{i=1}^N \varphi(Y_t^{i,\tau}) - \int_{\R} \rho_t^\tau(x) \varphi(x) \Id x \bigg)^2 \bigg),
  \end{align}
  where we used Minkowski's inequality in the third step. Now, by Proposition~\ref{prop: l_2_estimate_real_approx_trajectory_mean_field_trajectory} and
  \begin{equation*}
    |\varphi(X_t^{i,\tau}) - \varphi(Y_t^{i,\tau})|^2 \le \norm{\frac{\dd}{\dd x} \varphi}_{L^\infty}^2 |X_t^{i,\tau}-Y_t^{i,\tau}|^2,
  \end{equation*}
  we can estimate the first term by \( \norm{\frac{\dd}{\dd x} \varphi}_{L^\infty}^2  \frac{ 2 \norm{k_{\scriptscriptstyle{HK}}}_{L^2(\R)}^2 T}{(N-1)\tau} \exp\bigg( \frac{(C+\Lambda) T }{\tau} \bigg)\). For the second term we write out the square to obtain
  \begin{align*}
    & \frac{2}{N^2} \sum\limits_{i,j=1}^N \E \bigg(  \bigg(\varphi(Y_t^{i,\tau}) - \int_{\R} \rho_t^\tau(y) \varphi(y) \Id y   \bigg)  \bigg (\varphi(Y_t^{j,\tau}) - \int_{\R} \rho_t^\tau(y) \varphi(y) \Id y  \bigg)  \bigg)  \\
    &\quad= \frac{2}{N^2} \sum\limits_{i,j=1}^N \E \bigg( \varphi(Y_t^{i,\tau}) \varphi(Y_t^{j,\tau}) - \varphi(Y_t^{i,\tau}) \int_{\R} \rho_t^\tau(y) \varphi(y)  \Id y  \\
    &\quad\quad- \varphi(Y_t^{j,\tau}) \int_{\R} \rho_t^\tau(y) \varphi(y) \Id y + \bigg( \int_{\R} \rho_t^\tau(y)\varphi(y) \Id y \bigg)^2  \bigg).
  \end{align*}
  Now, using the fact that \(\rho_t^\tau\) is the conditional distribution of \(Y^{i,\tau}\) with respect to \(\cF_t^W\), we find
  \begin{align*}
    \E \bigg( \varphi(Y_t^{i,\tau}) \int_{\R} \rho_t^\tau(y) \varphi(y)  \Id y  \bigg)
    &= \E\bigg( \E\bigg( \varphi(Y_t^{i,\tau}) \int_{\R} \rho_t^\tau(y) \varphi(y)  \Id y \bigg| \cF_t^W \bigg)  \bigg) \\
    &= \E\bigg( \int_{\R} \rho_t^\tau(y) \varphi(y)  \Id y  \, \E( \varphi(Y_t^{i,\tau}) | \cF_t^W )  \bigg) \\
    &= \E\bigg( \bigg( \int_{\R} \rho_t^\tau(y) \varphi(y)  \Id y \bigg)^2  \bigg)  .
  \end{align*}
  Since \((Y_t^{i,\tau}, i = 1, \ldots, N)\) have identical conditional distributions given~\(\cF_t^W\), the same equality holds for~\(j \) instead of~\(i\). Additionally, using the fact that \(Y_t^{i,\tau},Y_t^{j,\tau}\) are conditionally independent for \(i \neq j \), we obtain
  \begin{align*}
    \E \big( \varphi(Y_t^{i,\tau}) \varphi(Y_t^{j,\tau}) \big)
    &= \E\big( \E( \varphi(Y_t^{i,\tau}) \varphi(Y_t^{j,\tau}) | \cF_t^W )\big) \\
    &= \E\big( \E( \varphi(Y_t^{i,\tau}) | \cF_t^W ) \E( \varphi(Y_t^{j,\tau}) | \cF_t^W ) \big) \\
    &= \E \bigg( \bigg( \int_{\R} \rho_t^\tau(y) \varphi(y)  \dd y \bigg)^2 \bigg)
  \end{align*}
  for the cross terms. Hence, the cross terms vanish and we can estimate the second term in~\eqref{eq: empirical_measure_estimate} by
  \begin{equation*}
    \frac{2}{N^2} \sum\limits_{i=1}^N \E  \bigg( \bigg(\varphi(Y_t^{i,\tau}) - \int_{\R} \rho_t^\tau(y) \varphi(y) \dd y \bigg)^2   \bigg)
    \le \frac{C(\norm{\phi}_{L^\infty(\R)} ) }{N}
  \end{equation*}
  for some finite constant \(C(\norm{\varphi}_{L^\infty(\R)} ) \), which depends only on \(\varphi\). Putting everything together, we find
  \begin{align*}
    \E(|\qv{\Pi_t^{N,\tau} , \varphi}- \qv{\rho_t^{\tau}, \varphi}|^2) &\le  \norm{\frac{\dd}{\dd x} \varphi}_{L^\infty}^2  \frac{ 2 \norm{k_{\scriptscriptstyle{HK}}}_{L^2(\R)}^2 T}{(N-1)\tau} \exp\bigg( \frac{(C+\Lambda) T }{\tau} \bigg)  +   \frac{C(\norm{\varphi}_{L^\infty(\R)} ) }{N} \\
    &\le C(\norm{k_{\scriptscriptstyle{HK}}}_{L^2(\R)}, T,\norm{\varphi}_{C^1(\R)})\bigg(  \frac{1}{N} + \frac{1}{(N-1)\tau} \exp\bigg( \frac{C T }{\tau} \bigg) \bigg).
  \end{align*}
  Next, using H{\"o}lder's inequality and Proposition~\ref{prop: difference of regularized mean-field SDE and the mean-field SDE}, we discover
  \begin{align*}
    &\E(|\qv{ \rho^{\tau}_t , \varphi } - \qv{ \rho_t , \varphi }|^2)
    \le \E(\norm{\varphi}_{L^2(\R)}^2 \norm{\rho^\tau_t- \rho_t}_{L^2(\R)}^2 )
    \le \norm{\varphi}_{L^2(\R)}^2   \norm{\rho^\tau - \rho}_{S^\infty_{\cF}([0,T];L^2(\R))}^2 \\
    &\quad\le \norm{\varphi}_{L^2(\R)}^2  C( \lambda, \Lambda,  T,\norm{k_{\scriptscriptstyle{HK}}}_{L^2(\R)} , \norm{\rho}_{S^\infty_{\cF^W}([0,T];L^2(\R))} \norm{\rho^\tau}_{S^\infty_{\cF^W}([0,T];L^2(\R))} )  \norm{k^\tau-k}_{L^2(\R)}^2.
  \end{align*}
  Therefore, combining this estimate with the previous one we obtain
  \begin{align*}
    &\E(|\qv{\Pi_t^{N,\tau}, \varphi} -  \qv{ \rho_t , \varphi }|^2) \\
    &\quad\le  \;  2 \E(|\qv{\Pi_t^{N,\tau} , \varphi}- \qv{\rho_t^{\tau}, \varphi}|^2)
    + 2 \E(|\qv{ \rho^\tau_t , \varphi } - \qv{ \rho_t , \varphi }|^2) \\
    &\quad\le  \;   C \Big(\lambda, \Lambda,  T, \norm{k_{\scriptscriptstyle{HK}}}_{L^2(\R)} , \norm{\varphi}_{C^1(\R)}, \norm{\varphi}_{L^2(\R)}, \norm{\rho}_{S^\infty_{\cF^W}([0,T];L^2(\R))},  \norm{\rho^\tau}_{S^\infty_{\cF^W}([0,T];L^2(\R))}
    \Big)  \\
    &\quad\quad \times \bigg(  \frac{1}{N} + \frac{1}{(N-1)\tau} \exp\bigg( \frac{(C+\Lambda) T }{\tau}  \bigg) +  \norm{k^\tau-k}_{L^2(\R)}^2 \bigg).
  \end{align*}
\end{proof}

\begin{remark}
  While we focus in the present section on the interaction force $k_{HK}(x)=\indicator{[0,R]}(|x|)x$, as used in the HK model, all results of Section~\ref{sec:mean field limit} extend verbatim to general interaction forces $k\in L^1(\R)\cap L^2(\R)$ such that there exists a sequence $(k^\tau)_{\tau \in \N}\subset \testfunctions{\R}$ satisfying the conditions (i)-(iii) in Section~\ref{subsec: HK model}.
\end{remark}

\bibliography{quellen}{}
\bibliographystyle{amsalpha}

\end{document}